\documentclass[11pt,reqno]{amsart}

\usepackage[T1]{fontenc}
\usepackage[utf8]{inputenc}
\usepackage{lmodern,amsmath,amssymb,amsthm,booktabs,microtype}
\usepackage{mathtools,array,longtable,enumitem}
\usepackage[hidelinks]{hyperref}

\hypersetup{
  pdftitle={Strict Hausdorff-dimension drop for quartic Salem Bernoulli convolutions},
  pdfsubject={Hausdorff dimension, Bernoulli convolutions, and rigorous explicit bounds},
  pdfkeywords={Bernoulli convolution, Salem number, Hausdorff dimension, Shannon entropy, Fourier transform, rigorous computation},
  pdfauthor={Guozheng Cheng, Xiang Fang, Xueqing Ma, Hongli Zhang}
}

\numberwithin{equation}{section}

\newtheorem{theorem}{Theorem}[section]
\newtheorem{proposition}[theorem]{Proposition}
\newtheorem{lemma}[theorem]{Lemma}
\newtheorem{corollary}[theorem]{Corollary}
\theoremstyle{definition}
\newtheorem{definition}[theorem]{Definition}
\theoremstyle{remark}

\newcommand{\R}{\mathbb R}
\newcommand{\Z}{\mathbb Z}

\newcommand{\E}{\mathbb E}
\newcommand{\Prob}{\mathbb P}
\DeclareMathOperator{\Law}{Law}

\newcommand{\dimH}{\dim_{\mathrm H}}

\makeatletter
\newcommand{\sectionnotoc}[1]{%
  \let\saved@tocwrite\@tocwrite
  \let\@tocwrite\@gobbletwo
  \section*{#1}%
  \let\@tocwrite\saved@tocwrite
}
\makeatother

\title[Quartic Salem Bernoulli convolutions]{Strict Hausdorff-dimension drop for\\
Quartic Salem Bernoulli convolutions}

\subjclass[2020]{Primary 28A80; Secondary 11R06, 60G50}

\keywords{Bernoulli convolution, Salem number, Hausdorff dimension,
Shannon entropy, Fourier transform, rigorous computation}

\date{}

\author{Guozheng Cheng}
\address{School of Mathematical Sciences, Dalian University of Technology, Dalian, 116024, P. R. China}
\email{gzhcheng@dlut.edu.cn}

\author{Xiang Fang}
\address{Department of Applied Mathematics, National Yang Ming Chiao Tung University, Hsinchu, 300093, Taiwan}
\email{xfang@nycu.edu.tw}

\author{Xueqing Ma}
\address{School of Mathematical Sciences, Dalian University of Technology, Dalian, 116024, P. R. China}
\email{maxueqing4680@mail.dlut.edu.cn}

\author{Hongli Zhang}
\address{School of Mathematical Sciences, Dalian University of Technology, Dalian, 116024, P. R. China}
\email{zhanghongli1@mail.dlut.edu.cn}

\begin{document}

\begin{abstract}
For every quartic Salem number $\beta$, we prove that the equal-weight
Bernoulli convolution $\nu_{\beta^{-1}}$ has Hausdorff dimension strictly
less than one. The main step converts large Fourier coefficients at
geometrically separated frequencies into a deficit in the Shannon
entropy of a finite convolution. A positive trigonometric product gives
the entropy estimate, while the algebraic norm controls the number of
distinct atoms in each spatial cell. For the two quartic Salem numbers
in $(1,2)$, reciprocal differences at suitable return times produce the
required Fourier coefficients. A finite block cover then gives the
dimension bound. This proof is independent of numerical computation.
Separate exact calculations give
$\dimH\nu_{\beta_1^{-1}}<1-1.4\cdot10^{-23}$ and
$\dimH\nu_{\beta_2^{-1}}<1-7.8\cdot10^{-22}$, where
$\beta_1<\beta_2$ are the two parameters in $(1,2)$.
The accompanying verification programs use rational interval arithmetic
and the Python standard library.
\end{abstract}

\maketitle

\vspace{-2em}

\tableofcontents

\section{Introduction}
\label{sec:introduction}
This paper is a sequel to~\cite{CFMZ}.
For $\beta>1$, let
\[
 W=\sum_{k=0}^{\infty}\varepsilon_k\beta^{-k},
 \qquad \nu_{\beta^{-1}}=\Law(W),
 \qquad
 \Prob(\varepsilon_k=1)=\Prob(\varepsilon_k=-1)=\frac12,
\]
where the signs are independent. The subscript denotes the contraction
ratio. Bernoulli convolutions form a basic family of self-similar
measures whose regularity depends on the arithmetic of this ratio.
For $\beta>2$, the elementary dimension bound
$\dimH\nu_{\beta^{-1}}\le\log2/\log\beta$ is already strictly below one.
The question considered here lies in the overlapping range $1<\beta<2$.

We study quartic Salem parameters. A Salem number is a real algebraic
integer $\beta>1$ whose other conjugates consist of $\beta^{-1}$ and
one or more pairs on the unit circle. In degree four there is exactly
one such pair, denoted by $\zeta,\bar\zeta$.
The contribution of this pair to an algebraic orbit is an oscillation
that does not decay. Controlling this term while keeping the cost of
the remaining orbit small is the main difficulty in the construction
below.

There are two closely related regularity results to distinguish.
In~\cite{CFMZ} we prove a strict drop in correlation
dimension for the two quartic Salem parameters in $(1,2)$.
 The preprint~\cite[Theorem~5.1]{OpenAI} establishes
singularity for the same parameters by a Fourier-analytic argument.
We emphasize that~\cite{OpenAI} is cited only for comparison:
neither the proof of Theorem~\ref{thm:main} nor the rigorous
computations underlying Theorem~\ref{thm:quantitative} rely on
any result from that preprint. 
Our localization argument starts from the reciprocal-orbit construction in~\cite{CFMZ}. 
Neither singularity nor a strict drop in correlation dimension alone
implies a strict drop in Hausdorff dimension. Our aim is to prove this
stronger dimension conclusion and to give explicit positive bounds for
the two deficits.

The distinction leads us to Shannon entropy. Garsia~\cite{Garsia}
related the entropy of finite Bernoulli convolutions to singularity,
and Akiyama, Feng, Kempton, and Persson~\cite{AFKP} developed matrix
methods for the entropy of algebraic Bernoulli convolutions.
Here we obtain a finite entropy deficit directly from Fourier
coefficients. The argument uses the exact distribution of the finite
sum, with coincident atoms identified, and ends with an explicit
Hausdorff cover.

For a Borel probability measure $\mu$ on $\R$, we use the convention
\[
 \dimH\mu
 =\inf\{\dimH A:A\subset\R\text{ is Borel and }\mu(A)=1\}.
\]
Translating and rescaling to digits $\{0,1\}$, or starting the series at
$k=1$, does not change this dimension.
The two quartic Salem numbers in $(1,2)$ are the larger real roots
$\beta_1<\beta_2$ of
\begin{equation}\label{eq:two-polynomials}
 \begin{split}
 P_1(X)&=X^4-X^3-X^2-X+1,\\
 P_2(X)&=X^4-2X^3+X^2-2X+1.
 \end{split}
\end{equation}
Equivalently,
\[
 u_1=\frac{1+\sqrt{13}}2,\qquad u_2=1+\sqrt2,
 \qquad
 \beta_i=\frac{u_i+\sqrt{u_i^2-4}}2.
\]
We include the elementary classification in
Proposition~\ref{prop:quartic-classification}.

\begin{theorem}\label{thm:main}
For every quartic Salem number $\beta$,
\[
 \dimH\nu_{\beta^{-1}}<1.
\]
\end{theorem}

In particular, each of the two measures corresponding to
\eqref{eq:two-polynomials} gives full mass to a Borel set of dimension
strictly below one and is therefore singular. The proof of
Theorem~\ref{thm:main} is independent of numerical computation. The following estimates
come from separate finite calculations.

\begin{theorem}\label{thm:quantitative}
For the two parameters in \eqref{eq:two-polynomials},
\[
 \dimH\nu_{\beta_1^{-1}}<1-1.4\cdot10^{-23},
 \qquad
 \dimH\nu_{\beta_2^{-1}}<1-7.8\cdot10^{-22}.
\]
\end{theorem}

The displayed numbers are lower bounds for the deficits
$1-\dimH\nu_{\beta_i^{-1}}$. The finite inequalities underlying Theorem~\ref{thm:quantitative}
are proved in the appendices and checked by the accompanying programs.

\subsection{The entropy estimate and the proof}

Write
\[
 X_N=\sum_{k=0}^{N-1}\varepsilon_k\beta^{-k},
 \qquad R=\frac{\beta}{\beta-1}.
\]
Let $A_N$ be the set of distinct values of $X_N$, and put
$p_x=\Prob(X_N=x)$ for $x\in A_N$. Its Shannon entropy is
\[
 H_N=H(X_N)=-\sum_{x\in A_N}p_x\log p_x.
\]
All logarithms are natural. We first prove that $M$ Fourier
coefficients of modulus at least $a$, at suitable geometrically
separated frequencies below the scale $\beta^N$, imply
\[
 H_N\le N\log\beta-\frac12Ma^2+2\log N+C_\beta.
\]
The gain comes from a positive trigonometric product. Its integral over
a bounded interval stays bounded, whereas the expected value of its
logarithm is large under the Bernoulli convolution. Relative entropy
then gives an estimate on spatial cells. The algebraic norm bounds the
number of distinct finite-convolution atoms in a cell by a constant
multiple of $N^2$, accounting for the term $2\log N$.

We next construct the Fourier coefficients needed to make the gain
larger than this logarithmic cost. The Fourier transform of
$\nu_{\beta^{-1}}$ is a product of cosines. At the frequencies used
below, the negative logarithm of its squared modulus is a sum of
the potential
$
 V(t)=-\log\cos^2(\pi t)
$
along a canonical algebraic orbit.

The unit-circle conjugates contribute a nondecaying oscillation to the
right tail of such an orbit. Removing that term produces a localized
sequence $H$ with two geometrically decaying real tails. We define
its action by
\[
 E=\sum_{j\in\Z}V(H_j).
\]
We start from fixed power seeds whose values at the unit-circle
conjugates are small. Reciprocal differences at suitable return times
then produce localized orbits with actions $E_k$, outgoing amplitudes
$\Delta_k$, and expanding coefficients $c_k$ satisfying
\[
 \Delta_k\longrightarrow0,\qquad
 E_k=o\!\left(\log\frac1{\Delta_k}\right),\qquad
 0<c_*\le|c_k|\le h<\frac1{100}.
\]
The constants $c_*$ and $h$ are independent of $k$. The initial
power seeds are specified in Lemma~\ref{lem:seed}.

Restoring the outgoing oscillation gives large Fourier coefficients
over a long interval of orbit indices. The preceding entropy estimate then yields
$H_N<N\log\beta$ at one fixed finite scale. Finally, an independent
block cover gives
\[
 \dimH\nu_{\beta^{-1}}\le\frac{H_N}{N\log\beta}<1.
\]
Thus the strict dimension bound follows from a finite entropy deficit;
the proof does not require an exact-dimensionality theorem or the
calculation of an infinite entropy rate.

The explicit dimension bounds are obtained from finite configurations
built from the same power seeds. Bounds on their action and outgoing
amplitudes yield the required finite inequalities, which are checked
using exact rational interval arithmetic. The large entropy block
lengths are handled symbolically: the verification uses the proved
estimates rather than enumerating the digit words.

Section~\ref{sec:entropy} establishes the entropy and covering estimates.
Sections~\ref{sec:algebra}--\ref{sec:hierarchy} construct the localized
orbits and control their action. Section~\ref{sec:fourier-closure}
uses these orbits to prove Theorem~\ref{thm:main}.
Appendices~\ref{app:explicit} and~\ref{app:scalar} prove the explicit
bounds, and Appendix~\ref{app:arithmetic} describes their exact
verification.

For the remainder of the proof, $\beta$ is one of $\beta_1,\beta_2$
unless stated otherwise, $\nu=\nu_{\beta^{-1}}$, and
$\widehat\nu(t)=\E e^{itW}$.

\section{Finite Fourier biases and Hausdorff covers}\label{sec:entropy}

We first turn a finite family of Fourier lower bounds into an upper
bound for $H_N$. We estimate entropy on a spatial grid, then use the
algebraic norm to pass from cells to the distinct atoms of $X_N$.

Throughout this section, $\beta>1$ is a quartic Salem number, with
conjugates $\beta,r=\beta^{-1},\zeta,\bar\zeta$. We write
$\nu=\nu_{\beta^{-1}}$ and $R=\beta/(\beta-1)$. The variables $W$, $X_N$
and the entropy $H_N$ are as in the introduction; equal atoms are merged
before entropy is computed.

\begin{lemma}\label{lem:entropy}
Let $Q>2$, $0<a\le1/2$, and let $M,N$ be positive integers. Suppose that
positive real frequencies satisfy
\[
 \xi_1\ge\pi,\qquad \xi_{j+1}\ge Q\xi_j,\qquad
 \beta^{-N}\xi_M\le1,\qquad
 |\widehat\nu(\xi_j)|\ge a\quad(1\le j\le M).
\]
Set
\[
 \Psi(a)=\frac{(1+a)\log(1+a)+(1-a)\log(1-a)}2,
 \qquad c_Q=\frac{Q-2}{Q-1},
\]
\[
\begin{aligned}
 D_{Q,a}
 &=\frac{2a}{\pi c_Q(1-(1+a)/Q)},\\
 B_{Q,a}
 &=\frac{aQ}{(1-a)(Q-1)},\\
 C_{\mathrm{sep}}
 &=1+\frac{1+2R}{2(\beta-1)},
\end{aligned}
\]
and
\[
 C_{\beta,Q}(a)=
 \log\!\bigl[(2R+2+D_{Q,a})C_{\mathrm{sep}}\bigr]+B_{Q,a}.
\]
Then
\begin{align}
 H_N
 &\le N\log\beta-M\Psi(a)+2\log N+C_{\beta,Q}(a)\notag\\
 &\le N\log\beta-\frac12Ma^2+2\log N+C_{\beta,Q}(a).
 \label{eq:entropy-parameter}
\end{align}
In particular, if $\beta^2>2$ and $Q=\beta^2$, there is a finite constant
$C_\beta>0$, independent of $a,M,N$, such that
\begin{equation}\label{eq:entropy}
 H_N\le N\log\beta-\frac12Ma^2+2\log N+C_\beta.
\end{equation}
\end{lemma}

\begin{proof}
Since $\nu$ is symmetric, its Fourier transform is real and
$\widehat\nu(\xi_j)=\mathbb E\cos(\xi_jW)$. We may therefore choose
$\sigma_j\in\{-1,1\}$ so that
\[
 \mathbb E[\sigma_j\cos(\xi_jW)]\ge a
 \qquad(1\le j\le M).
\]
Consider the positive trigonometric product
\[
 \mathcal P(x)=\prod_{j=1}^M(1+a\sigma_j\cos(\xi_jx)),
 \qquad F(x)=\log\mathcal P(x).
\]
Every factor lies between $1-a$ and $1+a$, so $F$ is well defined.
By concavity of the logarithm, for $-1\le t\le1$,
\[
 \log(1+at)\ge
 \frac{1+t}{2}\log(1+a)+\frac{1-t}{2}\log(1-a).
\]
The right-hand side is increasing in $t$. Taking
$t=\sigma_j\cos(\xi_jW)$ and using the choice of $\sigma_j$, we obtain
\[
 \mathbb E\log(1+a\sigma_j\cos(\xi_jW))
 \ge
 \frac{1+a}{2}\log(1+a)
 +\frac{1-a}{2}\log(1-a)
 =\Psi(a).
\]
Summing over $j$ gives $\mathbb EF(W)\ge M\Psi(a)$.
Moreover, $\Psi(0)=\Psi'(0)=0$ and
$\Psi''(a)=1/(1-a^2)\ge1$. Consequently,
\begin{equation*}
 \mathbb EF(W)\ge M\Psi(a)\ge\frac12Ma^2.
\end{equation*}

We next estimate the integral of $\mathcal P$ over a bounded interval.
Write each factor as
\[
 1+a\sigma_j\cos(\xi_jx)
 =1+\frac{a\sigma_j}{2}e^{i\xi_jx}
   +\frac{a\sigma_j}{2}e^{-i\xi_jx}.
\]
In the expansion of the product, consider a nonconstant term and let
$j$ be the largest index at which a nonconstant summand was chosen.
Its frequency $\omega$ satisfies
\[
 |\omega|\ge\xi_j-\sum_{i<j}\xi_i
 \ge\xi_j-\frac{\xi_j}{Q-1}
 =c_Q\xi_j>0,
\]
where we used $\xi_{i+1}\ge Q\xi_i$.
Thus no nonconstant term in the expansion has zero frequency, and
the constant term is one.

For a fixed largest index $j$, the sum of the absolute values of the
corresponding coefficients is $a(1+a)^{j-1}$. Indeed, the two
nonconstant summands in the $j$-th factor have total absolute coefficient
$a$, while each preceding factor contributes a total of $1+a$.
For any bounded interval $I$ and any $\omega\ne0$,
\[
 \left|\int_I e^{i\omega x}\,dx\right|\le\frac2{|\omega|}.
\]
Applying this estimate term by term, and using
$\xi_j\ge Q^{j-1}\xi_1$ and $\xi_1\ge\pi$, gives
\begin{align}
 \left|\int_I\mathcal P(x)\,dx-|I|\right|
 &\le\sum_{j=1}^M
       \frac{2a(1+a)^{j-1}}{c_Q\xi_j}\notag\\
 &\le\frac{2a}{c_Q\xi_1}
       \sum_{j=1}^M\left(\frac{1+a}{Q}\right)^{j-1}
 \le D_{Q,a}.
 \label{eq:product-integral}
\end{align}
Here the geometric series converges since $(1+a)/Q<1$.

We also need to control the variation of $F$ on short intervals.
Differentiating $F$ and using
$1+a\sigma_j\cos(\xi_jx)\ge1-a$, we find
\[
 |F'(x)|
 \le\frac{a}{1-a}\sum_{j=1}^M\xi_j.
\]
Since $\xi_j\le Q^{-(M-j)}\xi_M$, another geometric sum yields
\begin{equation}\label{eq:log-lipschitz}
 \operatorname{Lip}(F)
 \le\frac{a}{1-a}\sum_{j=1}^M\xi_j
 \le B_{Q,a}\xi_M.
\end{equation}

We now use these estimates to bound the entropy on a spatial grid.
Let $\eta=\beta^{-N}$, put
\(
 K=\lfloor W/\eta\rfloor,\
 C_k=[k\eta,(k+1)\eta),
\)
and let $U$ be the union of the cells for which
$\mathbb P(K=k)>0$.
Since $|W|\le R$ and $\eta\le1$, all these cells lie in
$[-R-1,R+1]$. As $\mathcal P$ is positive, applying
\eqref{eq:product-integral} to this containing interval gives
\[
 0<Z_U:=\int_U\mathcal P(x)\,dx
 \le\int_{-R-1}^{R+1}\mathcal P(x)\,dx
 \le2R+2+D_{Q,a}.
\]
For each occupied cell define
$ q_k=\frac1{Z_U}\int_{C_k}\mathcal P(x)\,dx.
$ 
These numbers are positive and sum to one.

Fix an occupied cell $C_k$ and let $y\in C_k$.
For every $x\in C_k$, we have $|x-y|\le\eta$.
The Lipschitz estimate \eqref{eq:log-lipschitz}, together with
$\eta\xi_M\le1$, therefore implies
\[
 F(x)\ge F(y)-B_{Q,a}\xi_M\eta
 \ge F(y)-B_{Q,a}.
\]
Integrating over $C_k$ gives
\[
 q_k\ge\frac{\eta}{Z_U}
       \exp\!\bigl(F(y)-B_{Q,a}\bigr).
\]
In particular, since $W\in C_K$ almost surely,
\[
 -\log q_K
 \le N\log\beta+\log Z_U-F(W)+B_{Q,a}.
\]
The nonnegativity of relative entropy between the distribution of
$K$ and the probability vector $(q_k)$ gives
$H(K)\le\mathbb E[-\log q_K]$. Taking expectations in the preceding
inequality and using the bounds for $\mathbb EF(W)$ and $Z_U$, we obtain
\begin{align}
 H(K)
 &\le N\log\beta-\mathbb EF(W)+\log Z_U+B_{Q,a}\notag\\
 &\le N\log\beta-M\Psi(a)
       +\log(2R+2+D_{Q,a})+B_{Q,a}.
 \label{eq:grid-entropy}
\end{align}

It remains to pass from the grid entropy to $H_N=H(X_N)$.
For this we use the separation of distinct atoms of $X_N$.
Take two distinct values $x,x'$ of $X_N$. Their difference can be
written as
\[
 x-x'=2u,\qquad
 u=\sum_{j=0}^{N-1}d_j\beta^{-j}\ne0,
 \qquad d_j\in\{-1,0,1\}.
\]
Since $\beta$ is an algebraic unit, $\beta^{-1}$ is an algebraic
integer, and hence so is $u$.
At the other three embeddings, corresponding to
$\beta^{-1},\zeta,\bar\zeta$, we have
\[
 \left|\sum_{j=0}^{N-1}d_j\beta^j\right|
 <\frac{\beta^N}{\beta-1},\qquad
 \left|\sum_{j=0}^{N-1}d_j\zeta^{-j}\right|\le N,
 \qquad
 \left|\sum_{j=0}^{N-1}d_j\bar\zeta^{-j}\right|\le N.
\]
The product of these three conjugate values with $u$ is its
algebraic norm, which is a nonzero integer. It follows that
$ 
 1\le
 |u|\frac{\beta^N}{\beta-1}N^2.
$ 
Since $|x-x'|=2|u|$, we conclude that
\begin{equation}\label{eq:norm-separation}
 |x-x'|\ge2(\beta-1)\beta^{-N}N^{-2}.
\end{equation}

Recall the natural decomposition
$ 
 W=X_N+\eta W',
$ 
where $W'$ has the same distribution as $W$ and satisfies
$|W'|\le R$.
On the event $K=k$, we have
$W\in[k\eta,(k+1)\eta)$, and therefore
$ 
 X_N\in\eta[k-R,k+1+R].
$ 
This interval has length $(1+2R)\eta$, while
\eqref{eq:norm-separation} shows that any two distinct atoms of
$X_N$ are separated by at least
$2(\beta-1)\eta N^{-2}$.
Consequently, the interval contains at most
$ 
 1+\frac{1+2R}{2(\beta-1)}N^2
 \le C_{\mathrm{sep}}N^2
$ 
distinct atoms of $X_N$.
Thus, conditional on any value of $K$, the variable $X_N$ has at
most $C_{\mathrm{sep}}N^2$ possible values. Its conditional entropy
is therefore bounded by
\[
 H(X_N\mid K)
 \le\log(C_{\mathrm{sep}}N^2)
 =2\log N+\log C_{\mathrm{sep}}.
\]
By the chain rule,
\[
 H_N=H(X_N)
 \le H(X_N,K)
 =H(K)+H(X_N\mid K).
\]
Combining this with \eqref{eq:grid-entropy} gives
\[
 H_N\le N\log\beta-M\Psi(a)+2\log N
 +\log\!\bigl[(2R+2+D_{Q,a})C_{\mathrm{sep}}\bigr]
 +B_{Q,a},
\]
which is the first inequality in \eqref{eq:entropy-parameter}.
The second follows from $\Psi(a)\ge a^2/2$.

Finally, suppose that $Q=\beta^2>2$.
Since $0<a\le1/2$, the constants above satisfy
\[
 D_{Q,a}\le\overline D_\beta
 :=\frac{1}{\pi c_Q(1-3/(2Q))},\qquad
 B_{Q,a}\le\overline B_\beta
 :=\frac{Q}{Q-1}.
\]
Both upper bounds depend only on $\beta$. Hence
\eqref{eq:entropy} holds with the fixed positive constant
$ 
 C_\beta=
 \log\!\bigl[(2R+2+\overline D_\beta)C_{\mathrm{sep}}\bigr]
 +\overline B_\beta.
$ 
This completes the proof.
\end{proof}

\begin{corollary}\label{cor:entropy-sharp}
Under the hypotheses of Lemma~\ref{lem:entropy}, suppose in addition
that $\beta>5/3$, $Q\ge25/9$, and $a\le1/9$. Then
$C_{\beta,Q}(a)<4$, and in particular
\begin{equation*}
 H_N\le N\log\beta-\frac12Ma^2+2\log N+4.
\end{equation*}
\end{corollary}

\begin{proof}
The stated ranges give
\[
 R<\frac52,\quad C_{\mathrm{sep}}<\frac{11}{2},\quad
 c_Q\ge\frac7{16},\quad
 1-\frac{1+a}{Q}\ge\frac35,
\]
and therefore
\[
 D_{Q,a}\le\frac{160}{189\pi},\qquad
 B_{Q,a}\le\frac{25}{128}.
\]
It follows that
\[
 C_{\beta,Q}(a)<
 \log\!\left[\left(7+\frac{160}{189\pi}\right)\frac{11}{2}\right]
 +\frac{25}{128}.
\]
For an entirely rational check, $\pi>3$ bounds the argument of this
logarithm by $45419/1134<41$, while
\[
 e^{15/4}>
 \sum_{j=0}^{8}\frac{(15/4)^j}{j!}
 =\frac{2459867141}{58720256}>41.
\]
The displayed constant is consequently less than
$15/4+25/128=505/128<4$.
\end{proof}

We next use independent blocks to turn an entropy bound at one fixed
scale into a bound for Hausdorff dimension.

\begin{lemma}\label{lem:block}
For every fixed positive integer $N$ and every $\beta>1$,
\begin{equation}\label{eq:block-dimension}
 \dim_{\mathrm H}\nu_{\beta^{-1}}
 \le\frac{H_N}{N\log\beta}.
\end{equation}
\end{lemma}

\begin{proof}
We group the original independent signs into consecutive blocks of
length $N$ and write
\[
 W=\sum_{j=0}^\infty\beta^{-Nj}B_j,
\]
where the $B_j$ are independent and identically distributed, each
with the merged distribution $(p_b)_{b\in A_N}$ of $X_N$.
Since $A_N$ is finite and $p_b>0$ for every $b\in A_N$, the
random variables $-\log p_{B_j}$ are integrable and have mean $H_N$.
The strong law of large numbers therefore gives
\[
 -\frac1m\sum_{j=0}^{m-1}\log p_{B_j}
 \longrightarrow H_N
 \qquad\text{almost surely}.
\]
Fix $\varepsilon>0$. For each $m\ge1$, define
\[
T_m
:=
\left\{
(b_0,\ldots,b_{m-1})\in A_N^m :
\prod_{j=0}^{m-1}p_{b_j}
\ge e^{-m(H_N+\varepsilon)}
\right\}.
\]
As the probabilities of all block words of length $m$ sum to one,
\[
 \#T_m\le e^{m(H_N+\varepsilon)}.
\]
Now fix a word $(b_0,\ldots,b_{m-1})\in T_m$.
The contribution of the first $m$ blocks to $W$ is
$\sum_{j=0}^{m-1}\beta^{-Nj}b_j$. The remaining tail is
$\beta^{-Nm}$ times a copy of $W$, whose absolute value is at most
$R\beta^{-Nm}$. Thus every infinite extension of this word
determines a point in the closed interval
\[
 \sum_{j=0}^{m-1}\beta^{-Nj}b_j
 +[-R\beta^{-Nm},R\beta^{-Nm}].
\]
Define
\[
 F_m=
 \bigcup_{(b_0,\ldots,b_{m-1})\in T_m}
 \left(
 \sum_{j=0}^{m-1}\beta^{-Nj}b_j
 +[-R\beta^{-Nm},R\beta^{-Nm}]
 \right).
\]
Then $F_m$ is a finite union of closed intervals, each of diameter
$2R\beta^{-Nm}$, and the number of these intervals is at most
$e^{m(H_N+\varepsilon)}$.

Define
\[
 F^{(q)}=\bigcap_{m\ge q}F_m,
 \qquad
 F=\bigcup_{q\ge1}F^{(q)}.
\]
These sets are Borel. By the strong law, almost every block sequence
has its length-$m$ prefix in $T_m$ for all sufficiently large $m$.
The point determined by such a sequence consequently belongs to
$F_m$ for all sufficiently large $m$, and hence to $F$.
Therefore $\nu(F)=1$.

Fix $q\ge1$ and take
\[
 t>\frac{H_N+\varepsilon}{N\log\beta}.
\]
For every $m\ge q$, we have $F^{(q)}\subset F_m$, so the intervals
defining $F_m$ give a cover of $F^{(q)}$. The sum of the $t$-powers
of their diameters is at most
\[
 e^{m(H_N+\varepsilon)}(2R\beta^{-Nm})^t
 =(2R)^t
 \exp\!\{m(H_N+\varepsilon-Nt\log\beta)\}.
\]
Since $H_N+\varepsilon-Nt\log\beta<0$, this expression tends to
zero as $m\to\infty$. The diameters of the covering intervals
also tend to zero. It follows that
$\mathcal H^t(F^{(q)})=0$.

As $F$ is the countable union of the sets $F^{(q)}$, we have
$\mathcal H^t(F)=0$. Together with $\nu(F)=1$, this gives
\[
 \dim_{\mathrm H}\nu_{\beta^{-1}}
 \le \dim_{\mathrm H}F
 \le\frac{H_N+\varepsilon}{N\log\beta}.
\]
Letting $\varepsilon\downarrow0$ proves
\eqref{eq:block-dimension}.
\end{proof}

\section{Quartic Salem parameters and algebraic seeds}
\label{sec:algebra}

We use the reciprocal construction and controlled returns from
\cite[Sections~4--5]{CFMZ}. The real embeddings give geometric tails,
while the unit-circle pair contributes an oscillatory term. We first
choose fixed power seeds and then use returns of the unit-circle
conjugates to reduce the oscillatory term.

\subsection{The two parameters in the overlapping range}

\begin{proposition}\label{prop:quartic-classification}
The quartic Salem numbers in $(1,2)$ are precisely the roots greater than
one of
\[
 P_1(X)=X^4-X^3-X^2-X+1,\qquad
 P_2(X)=X^4-2X^3+X^2-2X+1.
\]
\end{proposition}

\begin{proof}
The minimal polynomial of a quartic Salem number is a monic reciprocal
integer polynomial
\[
 P(X)=X^4-A_PX^3+B_PX^2-A_PX+1.
\]
Write its roots as $\beta,\beta^{-1},\zeta,\bar\zeta$, and put
$u=\beta+\beta^{-1}$ and $v=\zeta+\bar\zeta$. If $1<\beta<2$, then
$2<u<5/2$ and $-2<v<2$. Comparing coefficients gives
\[
 A_P=u+v,\qquad B_P=uv+2=A_Pu-u^2+2.
\]
Consequently $A_P\in\{1,2,3,4\}$. As $u$ varies in $(2,5/2)$, the
possible ranges for the second expression are
\[
 \begin{array}{c|cccc}
 A_P&1&2&3&4\\ \hline
 A_Pu-u^2+2&(-7/4,0)&(3/4,2)&(13/4,4)&(23/4,6).
 \end{array}
\]
The only integer possibilities are $(A_P,B_P)=(1,-1)$ and $(2,1)$.
Lemma~\ref{lem:roots} shows that both polynomials are irreducible
and have the required Salem roots in $(1,2)$.
\end{proof}

\begin{lemma}\label{lem:roots}
For $i=1,2$, the polynomial $P_i$ is irreducible over $\mathbb Q$.
Its roots are $\beta_i,r_i=\beta_i^{-1},\zeta_i,\bar\zeta_i$, where
$\operatorname{Im}\zeta_i>0$, and
\begin{equation*}
 \frac53<\beta_i<2,\qquad \frac12<r_i<\frac35,\qquad
 d_i:=|1-\zeta_i^2|>\frac32,\qquad |P_i'(\zeta_i)|>1.
\end{equation*}
Neither $\zeta_i$ nor $\zeta_i^2$ is a root of unity. The reciprocal
Laurent polynomials
\begin{equation*}
 U_1(X)=1+X+X^{-1},\qquad U_2(X)=X+X^{-1}
\end{equation*}
satisfy $U_i(\beta_i)=U_i(r_i)>0$ and
$0<|U_i(\zeta_i)|<1/2$.
\end{lemma}

\begin{proof}
The trace variable gives the identity
\[
 X^{-2}P(X)=(X+X^{-1})^2-A_P(X+X^{-1})+(B_P-2).
\]
The two trace roots are
\begin{equation}\label{eq:trace-roots}
\begin{aligned}
 (u,v)
 &=\left(\frac{1+\sqrt{13}}2,
         \frac{1-\sqrt{13}}2\right)
 &&\text{for }P_1,\\
 (u,v)
 &=(1+\sqrt2,1-\sqrt2)
 &&\text{for }P_2.
\end{aligned}
\end{equation}
Squaring the corresponding positive rational bounds gives
$34/15<u<5/2$ and $-2<v<2$ in both cases. Thus the four roots are
\[
 \beta=\frac{u+\sqrt{u^2-4}}2,\qquad
 r=\frac{u-\sqrt{u^2-4}}2,\qquad
 \zeta=\frac{v+i\sqrt{4-v^2}}2,\qquad \bar\zeta.
\]
Since $t+t^{-1}$ is strictly increasing for $t>1$, and its values at
$5/3$ and $2$ are $34/15$ and $5/2$, respectively, the stated bounds
for $\beta$ and $r$ follow.

The trace polynomials have nonsquare discriminants $13$ and $8$.
Thus $\mathbb Q(u)$ is a real quadratic field. Its other embedding sends
$u^2-4$ to $v^2-4<0$, so $u^2-4$ cannot be a square in
$\mathbb Q(u)$. The quadratic $X^2-uX+1$ is therefore irreducible over
$\mathbb Q(u)$. As $u=\beta+\beta^{-1}\in\mathbb Q(\beta)$, this proves
$[\mathbb Q(\beta):\mathbb Q]=4$, and hence the irreducibility of $P$.
In particular, $\beta$ is an algebraic unit.
If $\zeta$ were a root of unity, irreducibility would make $P$ divide
some $X^n-1$, contradicting $\beta>1$. The same conclusion holds for
$\zeta^2$.

For the first value of $v$ in \eqref{eq:trace-roots}, one has
$v<-5/4$ and $v^2=v+3<7/4$. For the second, $-1/2<v<0$, so again
$v^2<7/4$. It follows that
\[
 |1-\zeta^2|^2=4-v^2>\frac94.
\]
Differentiating
$P(X)=(X^2-uX+1)(X^2-vX+1)$ at $\zeta$ gives
\[
 |P'(\zeta)|=(u-v)\sqrt{4-v^2}=(u-v)d>1,
\]
because $u-v$ is $\sqrt{13}$ or $2\sqrt2$.
Finally,
\[
 U_1(\zeta_1)=\frac{3-\sqrt{13}}2,\qquad
 U_2(\zeta_2)=1-\sqrt2
\]
both lie in $(-1/2,0)$, whereas their values at $\beta_i$ and $r_i$
are $u+1$ and $u$, respectively.
\end{proof}

\subsection{Reciprocal orbits and power seeds}

We now separate the real and unit-circle contributions to the orbit.
Fix one of these polynomials and suppress the parameter index.
For every root $\lambda$ of $P$, put
\[
 \kappa_\lambda=\frac1{P'(\lambda)},\qquad
 T_m=\sum_{P(\lambda)=0}\kappa_\lambda\lambda^m
 \quad(m\in\mathbb Z).
\]
Lagrange interpolation for $1,X,X^2,X^3$, followed by comparison
of the coefficient of $X^3$, gives $T_0=T_1=T_2=0$ and $T_3=1$.
The sequence $(T_m)$ satisfies the recurrence defined by $P$.
Since the leading and constant coefficients of $P$ are both one,
the recurrence gives integer values in both directions. Hence
\begin{equation}\label{eq:trace-int}
 T_m\in\mathbb Z\quad(m\in\mathbb Z),\qquad
 \sum_{P(\lambda)=0}\kappa_\lambda\lambda^{i-1}K(\lambda)
 \in\mathbb Z
\end{equation}
for every $K\in\mathbb Z[X,X^{-1}]$ and $i\in\mathbb Z$.

Suppose that $K(\beta)\ne0$ and
\begin{equation}\label{eq:reciprocal-K}
 K(X^{-1})=\sigma X^{-2s_K}K(X),\qquad
 \sigma\in\{-1,1\},\quad s_K\in\mathbb Z.
\end{equation}
Irreducibility implies $K(\zeta)\ne0$: otherwise a monomial multiple
of $K$ would be divisible by $P$ and would also vanish at $\beta$.
Define
\begin{equation*}
 \begin{split}
 j_*&=4-2s_K,\qquad
 \xi_i=-\kappa_\beta\beta^{i-1}K(\beta),\\
 \chi_i&=2\operatorname{Re}
       \bigl(\kappa_\zeta\zeta^{i-1}K(\zeta)\bigr),\qquad
 \Delta=2|\kappa_\zeta K(\zeta)|>0 .
 \end{split}
\end{equation*}
Differentiation of the reciprocal identity for $P$ at a root yields
$P'(\lambda^{-1})=-\lambda^{-2}P'(\lambda)$.
In \eqref{eq:trace-int}, the expanding term is $-\xi_i$ and the
stable term is $\sigma\xi_{j_*-i}$. Hence
\begin{equation}\label{eq:reflection}
 \xi_i\equiv\sigma\xi_{j_*-i}+\chi_i\pmod{\mathbb Z},
 \qquad \chi_{j_*-i}=-\sigma\chi_i.
\end{equation}
For the second identity, one may directly use
\[
 \overline{\kappa_\zeta\zeta^{i-1}K(\zeta)}
 =-\sigma\kappa_\zeta\zeta^{j_*-i-1}K(\zeta).
\]
Write
\[
 \operatorname{can}(t)=t-\lfloor t+1/2\rfloor,\qquad
 Y_i=\operatorname{can}(\xi_i)\quad(i\in\mathbb Z).
\]
By construction, $Y_i\equiv\xi_i\pmod{\mathbb Z}$ for every
$i\in\mathbb Z$.

\begin{definition}\label{def:packet}
A \emph{reciprocal localized orbit} associated with $K$ has an integer
core $[a,b]$ satisfying $a+b=j_*$ and $\ell=b-a+1\ge5$, with
\[
 Y_i=\xi_i\quad(i<a),\qquad
 Y_i=\sigma\xi_{j_*-i}+\chi_i\quad(i>b).
\]
Its localization is
\[
 H_i=
 \begin{cases}
  Y_i,&i\le b,\\
  \sigma\xi_{j_*-i},&i>b.
 \end{cases}
\]
We require $H_i\in(-1/2,1/2)$ for every $i$.
Set
\[
 h_*=|\xi_a|,\qquad
 V(t)=-\log\cos^2(\pi t),\qquad
 E=\sum_{i\in\mathbb Z}V(H_i).
\]
The two localized real tails have absolute values $h_*r^j$, $j\ge1$.
A boundary distance of at least $\vartheta$ means
$|H_i|\le1/2-\vartheta$ for every $i$.
\end{definition}
The two sequences serve different purposes. The congruence
$Y_i\equiv\xi_i\pmod{\mathbb Z}$ is used to identify the
Fourier coefficients, while $H$ is used to control the action.
They agree for $i\le b$, and $Y_i=H_i+\chi_i$ for $i>b$.
We will later estimate the cost of restoring this oscillation.

We will use the following elementary bounds for the potential.
For $0\le t<1/2$, let
$\mathcal C(t)=\pi^2\sec^2(\pi t)$. On $[-t,t]$,
\begin{equation}\label{eq:potential-basic}
 V(u)\le\mathcal C(t)u^2,\qquad
 V''(u)\le2\mathcal C(t),\qquad
 |V'(u)|^2\le4\mathcal C(t)V(u).
\end{equation}
The first two follow from $V(0)=V'(0)=0$ and
$V''(u)=2\pi^2\sec^2(\pi u)$.
For the third, use
$V(u)=-\log(\cos^2(\pi u))\ge\sin^2(\pi u)$ and
$V'(u)=2\pi\tan(\pi u)$.
These bounds show that the action is finite. The boundary distance
is positive because the core is finite and the real tails decrease
geometrically. If that distance is at least $\vartheta$,
Cauchy--Schwarz gives
\begin{equation}\label{eq:core-derivative-bound}
 W_c:=\sum_{i=a}^b|V'(Y_i)|
 \le2\sqrt{\mathcal C(1/2-\vartheta)\,\ell E}.
\end{equation}

The orbit depends only on the values of $K$ at the roots of $P$.
Thus, for the purpose of evaluating a fixed orbit, we may replace
$K$ by any Laurent polynomial congruent to it modulo $P$.
The reciprocal identity \eqref{eq:reciprocal-K} is still understood
with respect to the original representative.
For any integer $s$, replacing $K$ by $X^sK$ shifts
the orbit indices by $s$ and changes the core to $[a-s,b-s]$.
The action, amplitude, real-tail coefficient, and boundary distance
remain unchanged. We will use these observations to compare different
representatives of the same seed.

The following fixed powers provide the starting orbits.

\begin{lemma}\label{lem:seed}
For $P_1$, define $K_{\rm s}:=U_1^3$, and for $P_2$, define
$K_{\rm s}:=U_2^4$. In each case, $K_{\rm s}$ admits a finite core
and defines a reciprocal localized orbit satisfying
\begin{equation*}
 0<h_*<\frac1{100},\qquad
 0<\Delta_{\rm s}<\frac14,\qquad E_{\rm s}<\infty,
\end{equation*}
with a positive global boundary distance $\vartheta_{\rm s}$.
\end{lemma}

\begin{proof}
Fix either polynomial and write $K_{\rm s}=U^m$, where $m=3$ or $4$
as specified. Since $U(X^{-1})=U(X)$, the reciprocal identity
\eqref{eq:reciprocal-K} holds with $\sigma=1$ and $s_K=0$, so
$j_*=4$. Moreover, Lemma~\ref{lem:roots} gives
$K_{\rm s}(\beta)=K_{\rm s}(r)>0$,
$|\kappa_\zeta|<1$, and $0<|U(\zeta)|<1/2$.
Thus the amplitude satisfies
\[
 0<\Delta_{\rm s}
 =2|\kappa_\zeta U(\zeta)^m|
 <2^{1-m}\le\frac14.
\]

We first show that every $\xi_i$ is irrational. Since
\[
 P'(\beta)=(\beta-r)|\beta-\zeta|^2>0,\qquad
 P'(r)=(r-\beta)|r-\zeta|^2<0,
\]
the definition of $\xi_i$ and the positivity of
$K_{\rm s}(\beta)=K_{\rm s}(r)$ imply that $\xi_i<0$, whereas
its image under the embedding of $\mathbb Q(\beta)$ sending
$\beta$ to $r$ is
\[
 -\kappa_r r^{i-1}K_{\rm s}(r)>0.
\]
As rational numbers are fixed by this embedding, $\xi_i$
cannot be rational. In particular,
$\xi_i\notin\mathbb Z+1/2$, and hence
$Y_i\in(-1/2,1/2)$ for every $i\in\mathbb Z$.

Since $|\xi_i|\to0$ as $i\to-\infty$, we may choose an integer
$a\le0$ sufficiently negative that
\[
 h_*:=|\xi_a|<\frac1{100},
\]
and set $b=4-a$. Then $a+b=j_*=4$ and
$b-a+1=5-2a\ge5$, as required for a finite core.

For $i<a$, we have
\[
 |\xi_i|=h_*r^{a-i}<\frac1{100}<\frac12,
\]
so $Y_i=\xi_i$. On the other hand, if $i>b$, then
$4-i=a-(i-b)$, and therefore
\[
 |\xi_{4-i}+\chi_i|
 \le h_*r^{i-b}+\Delta_{\rm s}
 <\frac1{100}+\frac14<\frac12.
\]
By the reflection identity \eqref{eq:reflection}, the quantity
$\xi_{4-i}+\chi_i$ is congruent to $\xi_i$ modulo $\mathbb Z$.
Since it lies in $(-1/2,1/2)$, it must equal $Y_i$.
This proves the two tail identities required in
Definition~\ref{def:packet}.

The resulting localized orbit has two real tails whose absolute
values are $h_*r^j$, $j\ge1$. Since the core is finite and
$|Y_i|<1/2$ for every $i$, it follows that
\[
 \vartheta_{\rm s}=
 \min\left\{
       \min_{a\le i\le b}(1/2-|Y_i|),\,
       1/2-h_*
     \right\}>0
\]
is a global boundary distance.

Finally, the action on the finite core is finite because
$Y_i\in(-1/2,1/2)$. For the two tails,
\eqref{eq:potential-basic} gives
\[
 \sum_{i<a}V(H_i)+\sum_{i>b}V(H_i)
 =2\sum_{j=1}^\infty V(h_*r^j)
 \le\frac{2\mathcal C(h_*)h_*^2r^2}{1-r^2}
 <\infty.
\]
Thus $E_{\rm s}<\infty$, completing the proof.
\end{proof}

\section{A reciprocal mirror and its action}
\label{sec:mirror}

For the remainder of the construction write $d=|1-\zeta^2|$.
Subtracting a distant translate produces two copies of the old core,
with a gap between them. A return of $\zeta^L$ towards one reduces the
outgoing oscillation. We estimate the action in these three regions
and in the two exterior tails. The geometric-sum estimate will be
used in the numerical specializations; the pointwise estimate
will suffice for the existence argument.

\begin{lemma}\label{lem:mirror}
Let a reciprocal localized orbit have core $[a,b]$, length $\ell$,
action $E$, amplitude $\Delta>0$, and boundary distance at least
$\vartheta>0$. Let its actual real-tail coefficient satisfy $h_*\le h$.
For an even integer $L>\ell$, set
\[
 g=L-\ell,\qquad
 \Delta^+=\Delta|1-\zeta^L|,\qquad
 \rho=hr^{g+1},\qquad U=\Delta+2hr.
\]
Suppose
\begin{equation}\label{eq:mirror-conditions}
 \Delta^++\rho\le\vartheta/2,\qquad
 U<1/2,\qquad \Delta^++hr<1/2.
\end{equation}
Then $K^+=(1-X^{-L})K$ defines a reciprocal localized orbit with
core $[a,b+L]$, amplitude $\Delta^+>0$, and actual real-tail
coefficient $h_*^+=(1-r^L)h_*$. Its boundary distance is at least
\begin{equation}\label{eq:margin-update}
 \min\{\vartheta-\Delta^+-\rho,\;1/2-U,\;1/2-hr\}.
\end{equation}
With $W_c$ as in \eqref{eq:core-derivative-bound}, put
\[
 A_\vartheta=2\pi^2\csc^2(\pi\vartheta/2),\qquad
 G_U=2\pi^2\sec^2(\pi U).
\]
The full localized action satisfies
\begin{align}
 E^+\le{}&
 2E+W_c(\Delta^++2\rho)
       +A_\vartheta\ell(\Delta^++\rho)^2 \notag\\
 &+G_U\left\{
 h^2gr^{g+1}+\frac{2\Delta hr}{1-r}
       +\Delta^2\left(\frac g4+\frac1{2d}\right)
       \right\}.                                               \label{eq:mirror-action}
\end{align}
It also satisfies the pointwise estimate
\begin{align}
 E^+\le{}&
 2E+W_c(\Delta^++2\rho)
       +A_\vartheta\ell(\Delta^++\rho)^2 \notag\\
 &+G_U\left\{
 h^2gr^{g+1}+\frac{2\Delta hr}{1-r}
                       +\frac12g\Delta^2
       \right\}.                                               \label{eq:mirror-coarse}
\end{align}
\end{lemma}

\begin{proof}
The reciprocal identity for $K^+$ has sign $-\sigma$, center
parameter $s_K-L/2$, and reflection index $j_*+L$. Its modes satisfy
\begin{equation}\label{eq:mirror-modes}
 \xi_i^+=\xi_i-\xi_{i-L}=(1-r^L)\xi_i,\qquad
 \chi_i^+=\chi_i-\chi_{i-L}.
\end{equation}
In particular, $K^+(\beta)\ne0$, and the new amplitude is
$\Delta^+$ because $|1-\zeta^{-L}|=|1-\zeta^L|$.
It is positive since $\zeta$ is not a root of unity.

On the first core, $a\le i\le b$, the value
$Y_i-\xi_{i-L}$ is congruent to $\xi_i^+$ and differs from $Y_i$
by at most $\rho$. Condition \eqref{eq:mirror-conditions} keeps it in the canonical
interval.
On the second core, write $i=L+j_*-t$, where
$a\le t\le b$. The reflection identity for $K^+$,
whose sign is $-\sigma$, gives
$\chi_i^+=\sigma\chi_t^+$ and
\[
\begin{aligned}
 \xi_i^+
 &\equiv-\sigma\xi_t^++\chi_i^+\\
 &\equiv\sigma(-Y_t+\xi_{t-L}+\chi_t^+)
 \pmod{\mathbb Z}.
\end{aligned}
\]
The last expression differs from $-\sigma Y_t$ by
at most $\rho+\Delta^+\le\vartheta/2$.
It therefore lies in the canonical interval, and
\[
 Y_i^+=\sigma(-Y_t+\xi_{t-L}+\chi_t^+).
\]
The segments joining these values to the old core values lie in
$[-1/2+\vartheta/2,1/2-\vartheta/2]$, where
$V''\le A_\vartheta$.
Taylor's theorem and the evenness of $V$ bound the action of the two
cores by
\begin{equation}\label{eq:mirror-core-action}
 2\sum_{i=a}^bV(Y_i)+W_c(\Delta^++2\rho)
                       +A_\vartheta\ell(\Delta^++\rho)^2.
\end{equation}
For the quadratic term we used
$\tfrac12[\rho^2+(\rho+\Delta^+)^2]\le(\rho+\Delta^+)^2$.

In the gap, write $i=b+j$ with $1\le j\le g$.
The trace congruence identifies the canonical value with
$x_j+z_j+\chi_{b+j}$, where
\[
 x_j=\sigma\xi_{a-j},\qquad z_j=-\xi_{b+j-L},
 \qquad |x_j|=h_*r^j,\quad |z_j|=h_*r^{g+1-j}.
\]
The sum of these absolute values and $|\chi_{b+j}|$ is at most
$U<1/2$. For $|x|+|z|+|w|\le U$, integrating the mixed derivative
for the first two variables and then applying Taylor's theorem in
the third gives
\[
 V(x+z+w)-V(x)-V(z)
 \le G_U\left(|xz|+(|x|+|z|)|w|+\frac12w^2\right).
\]
The integration segments stay in $[-U,U]$, where
$|V'(t)|\le G_U|t|$. Summing the first two terms gives the bounds
\[
 h^2gr^{g+1},\qquad \frac{2\Delta hr}{1-r},
\]
respectively.
Writing $\zeta=e^{i\phi}$ and
$\chi_{b+j}=\Delta\cos(\theta+j\phi)$, the geometric sum gives
\[
 \begin{split}
 \frac12\sum_{j=1}^g\chi_{b+j}^2
 &=\Delta^2\left\{\frac g4+
       \frac14\operatorname{Re}
           \left(e^{2i\theta}\sum_{j=1}^g\zeta^{2j}\right)\right\}\\
 &\le\Delta^2\left(\frac g4+\frac1{2d}\right).
 \end{split}
\]
Alternatively, the same term is at most $g\Delta^2/2$.

The left exterior tail is $(1-r^L)\xi_i$ and is canonical.
Beyond $b+L$, the new reflection relation gives
\[
 Y_i^+=-\sigma\xi^+_{j_*+L-i}+\chi_i^+,
\]
which is canonical by $\Delta^++hr<1/2$.

The two new exterior localized tails have absolute values
$h_*^+r^j$, $j\ge1$. To account for the baseline action, write
\[
 E_{\mathrm{core}}=\sum_{i=a}^b V(Y_i),
 \qquad T=\sum_{j\ge1}V(h_*r^j),
 \qquad E=E_{\mathrm{core}}+2T.
\]
The two cores contribute $2E_{\mathrm{core}}$ before the
Taylor errors in \eqref{eq:mirror-core-action}.
In the gap, the base terms $V(x_j)+V(z_j)$ sum to at most $2T$.
The two new exterior tails contribute at most $2T$, since
$h_*^+\le h_*$. Thus the baseline contribution is at most
\[
 2E_{\mathrm{core}}+2T+2T=2E.
\]
Combining these bounds with \eqref{eq:mirror-core-action} and the
two square-sum estimates proves \eqref{eq:mirror-action} and
\eqref{eq:mirror-coarse}. The bounds in each region give
\eqref{eq:margin-update}.
\end{proof}

\section{Controlled returns and a low-action hierarchy}
\label{sec:hierarchy}

We now choose successive mirrors so that the amplitude decreases
faster than the action grows. The delays must separate the two copies
of each core while keeping the action in the gap under control.

Fix one of the seeds in Lemma~\ref{lem:seed}, including its core.
Let $\ell_{\rm s}$ denote the core length and set $h=h_*$.
Keep these choices fixed, and put
$\omega=\zeta^2$ and $\operatorname{err}(q)=|1-\omega^q|$.
For each positive integer $R$, let $q(R)$ be the least integer
attaining the minimum of $\operatorname{err}(q)$ over $1\le q\le R$.

\begin{lemma}\label{lem:return}
For every positive integer $R$,
\[
 0<\operatorname{err}(q(R))\le\frac{2\pi}{R},
 \qquad q(R)\longrightarrow\infty\quad(R\longrightarrow\infty).
\]
If
$\operatorname{err}(q(R))<\operatorname{err}(q(R'))$, then
$q(R)>R'\ge q(R')$. Moreover, for positive integers $m,q$,
\begin{equation}\label{eq:return-multiple}
 |1-\omega^{mq}|\le m\operatorname{err}(q).
\end{equation}
\end{lemma}

\begin{proof}
Write $\omega=e^{2\pi i\alpha}$ with $\alpha$ irrational.
Divide $[0,1)$ into $R$ equal half-open intervals. Among the
$R+1$ fractional parts of $0,\alpha,\ldots,R\alpha$, two must
lie in the same interval. Taking the difference of their indices,
we obtain integers $1\le q\le R$ and $n$ such that
$|q\alpha-n|\le1/R$. Since $q(R)$ minimizes the error over
$1\le q\le R$, it follows that
\[
 \operatorname{err}(q(R))
 \le\operatorname{err}(q)
 =|1-e^{2\pi i(q\alpha-n)}|
 \le2\pi|q\alpha-n|
 \le\frac{2\pi}{R}.
\]
Every error is positive because $\alpha$ is irrational.
In particular, the minimum error over any fixed finite set of
positive integers is strictly positive. Since
$\operatorname{err}(q(R))\to0$, we must have $q(R)\to\infty$.

Now suppose that
$\operatorname{err}(q(R))<\operatorname{err}(q(R'))$.
By the definition of $q(R')$, no integer $q\le R'$ can have
an error smaller than $\operatorname{err}(q(R'))$.
Hence $q(R)>R'\ge q(R')$.

Finally, the geometric-sum identity gives
\[
 1-\omega^{mq}
 =(1-\omega^q)\sum_{j=0}^{m-1}\omega^{jq}.
\]
Since $|\omega|=1$, taking absolute values yields
\[
 |1-\omega^{mq}|
 \le m|1-\omega^q|
 =m\operatorname{err}(q),
\]
which proves \eqref{eq:return-multiple}.
\end{proof}

Set
\[
 \eta=\min\{\vartheta_{\rm s}/8,1/100\}>0.
\]
We first make one mirror to obtain a sufficiently small amplitude.
Choose $R_0$ large, put $q_0=q(R_0)$ and $L_0=2q_0$, and require
$q_0>\ell_{\rm s}$. By Lemma~\ref{lem:return}, we may also require
\begin{align*}
 \Delta_0:=\Delta_{\rm s}\operatorname{err}(q_0)
     &<\min\{e^{-32},\eta/4\},                                  \\
 \Delta_0+hr^{2q_0-\ell_{\rm s}+1}
     &\le\vartheta_{\rm s}/2,                                  \\
 \frac{hr^{\ell_0}}{1-r^{\ell_0}}
     &<\eta/4,\qquad
 \ell_0=\ell_{\rm s}+2q_0<3q_0<4q_0.                            
\end{align*}
These conditions hold simultaneously for large $R_0$, since
$q_0\to\infty$ and its error tends to zero. Moreover,
\[
 U_{\rm s}=\Delta_{\rm s}+2hr<\frac{27}{100},
 \qquad \Delta_0+hr<\frac12.
\]
Thus Lemma~\ref{lem:mirror} applies.
The new boundary distance is at least $4\eta$. On the two core
copies the distance is at least $\vartheta_{\rm s}/2\ge4\eta$;
in the gap and exterior real tails it exceeds $23/100$ and
$49/100$, respectively.
Fix this prepared orbit and index it by $k=0$.
Its action $E_0$ is finite, and its actual real-tail coefficient
is at most $h$.

\begin{proposition}\label{prop:hierarchy}
Starting from the prepared orbit, there is a sequence of reciprocal
localized orbits with a common left endpoint, real-tail coefficients
at most $h$, and global boundary distances at least $3\eta$.
For a finite constant $C_E$ independent of $k$,
\begin{equation}\label{eq:hierarchy-goal}
 0<\Delta_k<e^{-32\cdot2^k},\qquad
 x_k:=\log(1/\Delta_k)\ge26(5/2)^k,\qquad
 E_k\le C_E2^k .
\end{equation}
In particular, $E_k=o(x_k)$.
\end{proposition}

\begin{proof}
We construct the orbits inductively. Set $m_0=1$ and
$\Delta_{-1}=\Delta_{\rm s}$, so that the preparation step satisfies
$\Delta_0=\Delta_{-1}|1-\omega^{m_0q_0}|$.
Suppose that the orbit at level $k-1$ has been constructed.
For $k\ge1$, choose
\[
 m_k=4^k,\qquad
 R_k=\left\lfloor\frac{\Delta_{k-1}^{-3/2}}{2m_k}\right\rfloor,
 \qquad q_k=q(R_k),\qquad L_k=2m_kq_k.
\]
The intended new amplitude and core length are
\[
 \Delta_k=\Delta_{k-1}|1-\omega^{m_kq_k}|,
 \qquad \ell_k=\ell_{k-1}+L_k.
\]
The exponent $3/2$ in the definition of $R_k$ balances the gap cost against the decrease in amplitude. Writing $\delta=\Delta_{k-1}$, the estimates below give
\[
L_k\delta^2\le\delta^{1/2},
\qquad
\Delta_k<32m_k^2\delta^{5/2}.
\]
The first bound controls the quadratic gap cost. The second leads to growth of $\log(1/\Delta_k)$ on the $(5/2)^k$ scale, faster than the $2^k$ bound for the action. The choice $m_k=4^k$, together with the increase of the return denominators proved below, ensures the required separation of successive cores.

We first establish estimates that will allow us to apply
Lemma~\ref{lem:mirror}.

Let $\vartheta_j$ denote the global boundary distance at level $j$,
and write
\[
 \rho_j=hr^{L_j-\ell_{j-1}+1},\qquad
 S_j=\sum_{i=1}^j(\Delta_i+\rho_i),\qquad S_0=0.
\]
We maintain the inductive bounds
\begin{equation}\label{eq:hierarchy-invariants}
 \Delta_{k-1}\le\Delta_0^{2^{k-1}}<e^{-32\cdot2^{k-1}},
 \qquad \ell_{k-1}<4^kq_{k-1},
 \qquad \vartheta_{k-1}\ge4\eta-S_{k-1}.
\end{equation}
They hold for $k=1$ by the preparation step.

Put $\delta=\Delta_{k-1}$ and $m=m_k$.
The amplitude bound in \eqref{eq:hierarchy-invariants} gives
\[
 \log(32m^2\sqrt{\delta})
 <\log32+k\log16-16\cdot2^{k-1}
 <4+3k-16k<0,
\]
where we used $2^{k-1}\ge k$ for $k\ge1$. Thus
\begin{equation}\label{eq:key-small}
 32m^2\sqrt{\delta}<1.
\end{equation}
In particular, $\delta^{-3/2}/(2m)>2$, and hence
\[
 R_k\ge\frac{\delta^{-3/2}}{4m}\ge1.
\]
Lemma~\ref{lem:return} now gives
\begin{equation}\label{eq:controlled-return}
 \operatorname{err}(q_k)
 \le\frac{2\pi}{R_k}
 \le8\pi m\delta^{3/2}
 <32m\delta^{3/2},
 \qquad L_k\le\delta^{-3/2},
\end{equation}
where the length bound follows from $q_k\le R_k$.

We next show that the new return denominator exceeds the preceding
one. By the amplitude relation at level $k-1$, including the
preparation step when $k=1$, and \eqref{eq:return-multiple},
\[
 \delta
 \le m_{k-1}\Delta_{k-2}
                  \operatorname{err}(q_{k-1}).
\]
Since $\Delta_{k-2}<1$, this implies
$\operatorname{err}(q_{k-1})>\delta/m_{k-1}$.
On the other hand, \eqref{eq:key-small} and
\eqref{eq:controlled-return} give
\[
 \operatorname{err}(q_k)
 <32m\delta^{3/2}
 <\frac{\delta}{m_{k-1}}
 <\operatorname{err}(q_{k-1}).
\]
By Lemma~\ref{lem:return}, we have $q_k>q_{k-1}$.
Combining this with
$\ell_{k-1}<4^kq_{k-1}=m q_{k-1}$, we obtain
\[
 L_k>2\ell_{k-1},\qquad
 \ell_k<3m_kq_k<4^{k+1}q_k,\qquad
 \ell_k>3\ell_{k-1}.
\]
In particular, the delay $L_k$ is large enough to separate the
two copies of the preceding core.

We also obtain a stronger bound for the new amplitude.
By \eqref{eq:return-multiple} and \eqref{eq:controlled-return},
\[
 0<\Delta_k
 =\delta|1-\omega^{mq_k}|
 \le m\delta\operatorname{err}(q_k)
 <32m^2\delta^{5/2}
 <\delta^2.
\]
The positivity follows because $\omega$ is not a root of unity,
and the last inequality follows from \eqref{eq:key-small}.
Consequently,
\[
 \Delta_k<\Delta_0^{2^k}<e^{-32\cdot2^k},
\]
which establishes the amplitude bound at the next level.

The sharper amplitude estimate also gives
\[
 x_k\ge\frac52x_{k-1}-\log32-k\log16.
\]
Dividing by $(5/2)^k$ and summing, we find
\[
 \begin{aligned}
 \frac{x_k}{(5/2)^k}
 &\ge x_0-
       \sum_{j=1}^k\frac{\log32+j\log16}{(5/2)^j}\\
 &>x_0-4\cdot\frac23-3\cdot\frac{10}{9}
   =x_0-6>26,
 \end{aligned}
\]
where we used $x_0>32$. This proves the second bound in
\eqref{eq:hierarchy-goal}.

It remains to verify the boundary conditions before applying the
mirror. Put $g_k=L_k-\ell_{k-1}$, so that
$g_k>\ell_{k-1}$ and $\rho_k=hr^{g_k+1}$.
Since $\ell_j>3\ell_{j-1}$, we have
$g_j+1\ge3^{j-1}\ell_0$ for $1\le j\le k$.
Therefore, 
\[
 \sum_{j=1}^k\rho_j
 \le h\sum_{j=1}^k r^{3^{j-1}\ell_0}
 \le\frac{hr^{\ell_0}}{1-r^{\ell_0}}
 <\frac{\eta}{4},
\]
where the last inequality was arranged in the preparation step.
Similarly, the amplitude bounds and $\Delta_0<1/2$ give
\[
 \sum_{j=1}^k\Delta_j
 \le\sum_{j=1}^\infty \Delta_0^{2^j}
 \le\frac{\Delta_0^2}{1-\Delta_0^2}
 <\Delta_0<\frac{\eta}{4}.
\]
Thus, $S_k<\eta/2$.

By the inductive boundary bound,
\[
 \vartheta_{k-1}\ge4\eta-S_{k-1}>\frac{7\eta}{2},
 \qquad
 \Delta_k+\rho_k<\frac{\eta}{2}
 <\frac{\vartheta_{k-1}}2.
\]
Moreover, since $\delta<e^{-32}$ and $h<1/100$,
\[
 U_k=\delta+2hr<\frac3{100},
 \qquad
 \Delta_k+hr<\frac2{100}.
\]
Together with $L_k>2\ell_{k-1}$, these inequalities verify
all the hypotheses of Lemma~\ref{lem:mirror}.
We may therefore form the orbit at level $k$.

The mirror preserves the left endpoint and does not increase the
actual real-tail coefficient, which remains at most $h$.
By \eqref{eq:margin-update}, the new boundary distance satisfies
\[
 \vartheta_k\ge
 \min\left\{
 4\eta-S_k,\,
 \frac12-U_k,\,
 \frac12-hr
 \right\}.
\]
The last two quantities exceed $47/100$ and $49/100$,
respectively. Since $\eta\le1/100$, it follows that
\[
 \vartheta_k\ge4\eta-S_k
 >\frac{7\eta}{2}>3\eta.
\]
This establishes the boundary bound and completes the inductive
construction.

We now estimate the action. Define the fixed finite constants
\[
 \begin{aligned}
 S_r&=\frac{r}{(1-r)^2},
 & B_0&=2\pi\cot(\pi\eta),\\
 A_0&=2\pi^2\csc^2(\pi\eta),
 & G_0&=2\pi^2\sec^2(3\pi/100).
 \end{aligned}
\]
At level $k-1$, the boundary distance is greater than $3\eta$.
Thus, in the notation of Lemma~\ref{lem:mirror},
\[
 W_c\le B_0\ell,\qquad
 A_\vartheta\le A_0,\qquad G_U\le G_0,
\]
where $\ell=\ell_{k-1}$ and $U=U_k$.
Here, the first estimate follows from
$|V'(t)|=2\pi|\tan(\pi t)|$ on the old core.

Write $g=g_k$. Since
$\ell<L_k/2\le\delta^{-3/2}/2$ and
$\Delta_k<\delta^2$, we have
\[
 \ell\Delta_k\le\frac12\sqrt\delta<1,\qquad
 \ell\Delta_k^2<1,\qquad
 g\delta^2\le\sqrt\delta<1.
\]
Also, $g>\ell$ and $\rho_k=hr^{g+1}$ imply
\[
 \ell\rho_k\le hS_r,\qquad
 \ell\rho_k^2\le h^2S_r,\qquad
 gr^{g+1}\le S_r,
\]
using $S_r=\sum_{n=1}^\infty nr^n$.

Substituting these estimates into \eqref{eq:mirror-coarse},
and using $(u+v)^2\le2(u^2+v^2)$, gives
\[
 E_k\le2E_{k-1}+C_*,
\]
where
\[
 C_*=
 B_0(1+2hS_r)+2A_0(1+h^2S_r)
 +G_0\left(
 h^2S_r+\frac{2hr}{1-r}+\frac12
 \right).
\]
All the quantities defining $C_*$ are fixed, so $C_*$
is finite and independent of $k$. Iterating the inequality yields
\[
 E_k\le2^kE_0+(2^k-1)C_*
 \le(E_0+C_*)2^k.
\]
Taking $C_E=E_0+C_*$ proves the remaining bound in
\eqref{eq:hierarchy-goal}. Finally,
\[
 0\le\frac{E_k}{x_k}
 \le\frac{C_E}{26}\left(\frac45\right)^k
 \longrightarrow0,
\]
and hence $E_k=o(x_k)$.
\end{proof}

We next show that the expanding coefficients remain bounded away
from zero, as needed for the Fourier frequencies in the next section.
Every mirror preserves the left endpoint $a$. We shift the indices
by writing $i=a+j$ and set $c_k=\xi_a^{(k)}$, with
$c_0=\xi_a^{(0)}$ for the prepared orbit.
By \eqref{eq:mirror-modes}, each mirror multiplies the expanding
coefficient by $1-r^{L_j}$. Hence,
\begin{equation}\label{eq:c-product}
 c_k=c_0\prod_{j=1}^k(1-r^{L_j}),\qquad
 0<|c_0|\le h<1/100.
\end{equation}
The delays are strictly increasing, since
$L_j>2\ell_{j-1}$ and $\ell_{j-1}>L_{j-1}$.
As they are distinct positive integers,
\[
 \sum_{j=1}^\infty r^{L_j}
 \le\sum_{n=1}^\infty r^n=\frac{r}{1-r}.
\]
Moreover, for $0\le t\le r$,
\[
 -\log(1-t)=\int_0^t\frac{du}{1-u}
 \le\frac{t}{1-r}.
\]
Applying this to the product in \eqref{eq:c-product} gives
\[
 \log\frac{|c_0|}{|c_k|}
 =\sum_{j=1}^k-\log(1-r^{L_j})
 \le\frac{1}{1-r}\sum_{j=1}^k r^{L_j}
 \le\frac{r}{(1-r)^2}.
\]
Consequently,
\begin{equation}\label{eq:c-bound}
 0<c_*:=|c_0|\exp\left(-\frac{r}{(1-r)^2}\right)
 \le|c_k|\le h<1/100.
\end{equation}

After shifting the common left endpoint to zero, we retain the
notation $Y_j^{(k)},H_j^{(k)},\chi_j^{(k)}$ for the orbit, its
localization, and its oscillatory term. The core is now
$[0,\ell_k-1]$. Since $\xi_{a+j}^{(k)}=c_k\beta^j$,
the canonical definition gives
\begin{equation}\label{eq:global-congruence}
 Y_j^{(k)}=\operatorname{can}(c_k\beta^j)
          \equiv c_k\beta^j\pmod{\mathbb Z}
 \qquad(j\in\mathbb Z).
\end{equation}
The localization identities take the form
\begin{equation}\label{eq:localized-structure}
 \begin{gathered}
 Y_j^{(k)}
   =H_j^{(k)}+\mathbf1_{\{j\ge\ell_k\}}\chi_j^{(k)},\qquad
 |\chi_j^{(k)}|\le\Delta_k,\\
 |H_j^{(k)}|\le hr^{j-\ell_k+1}\quad(j\ge\ell_k),\qquad
 E_k=\sum_{j\in\mathbb Z}V(H_j^{(k)}).
 \end{gathered}
\end{equation}
In particular, the real right tail is bounded by $hr$ at its
first index $j=\ell_k$ and decays geometrically thereafter.
Finally, the uniform lower bound $|c_k|\ge c_*>0$ implies
\[
 \min\{s\in\mathbb Z_{\ge0}:|c_k|\beta^s\ge1\}
 \le\left\lceil\frac{\log(1/c_*)}{\log\beta}\right\rceil.
\]
Thus these integers are bounded independently of $k$.

\section{Localized trajectories and the entropy deficit}
\label{sec:fourier-closure}

We now restore the outgoing oscillation and apply the entropy estimate
from Section~\ref{sec:entropy}. The argument uses the full localized
action, a uniform bound for the real tail, and an expanding left
coefficient bounded away from zero. These are also the quantities
provided by the finite constructions in Appendix~\ref{app:explicit}.

For the centered trajectories constructed above, the required
congruence already follows from \eqref{eq:global-congruence}.
The following lemma gives an alternative criterion based on an
integer recurrence and an exact left tail. It also applies to
the Green-sequence construction in Appendix~\ref{app:p2}.

\begin{lemma}
\label{lem:recurrence-congruence}
Let $P$ be the monic integral quartic polynomial of $\beta$, and let
$(Sy)_j=y_{j+1}$. Suppose that a real bilateral sequence $y$ satisfies
$P(S)y\in\mathbb Z^{\mathbb Z}$
and $y_j=c\beta^j$ for every $j<0$, where $c\ne0$. Then
\begin{equation}\label{eq:recurrence-congruence}
 y_j\equiv c\beta^j\pmod{\mathbb Z}\qquad(j\in\mathbb Z).
\end{equation}
\end{lemma}

\begin{proof}
Write $P(X)=X^4+\sum_{i=0}^3p_iX^i$, with $p_i\in\mathbb Z$, and set
$z_j=y_j-c\beta^j$. The identity $P(\beta)=0$ gives
\[
 z_{j+4}=(P(S)y)_j-\sum_{i=0}^3p_i z_{j+i}.
\]
Since $z_{-4}=z_{-3}=z_{-2}=z_{-1}=0$, we may start at $j=-4$.
Induction gives $z_j\in\mathbb Z$ for every subsequent index.
\end{proof}

The congruence \eqref{eq:recurrence-congruence} identifies the Fourier
coefficients that we will use. For every integer $n$, the Bernoulli
product and the periodicity of $V$ give
\begin{equation}\label{eq:fourier-product}
 |\widehat\nu(\pi|c|\beta^n)|^2
 =\prod_{j\le n}\cos^2(\pi c\beta^j)
 =\exp\!\left(-\sum_{j\le n}V(y_j)\right).
\end{equation}
Here, $V(u)=-\log\cos^2(\pi u)$ is understood as an extended nonnegative
function. For each such $n$, the infinite left tail converges because
$c\beta^j\to0$ geometrically as $j\to-\infty$.
For the centered trajectories constructed above, the required congruence
is \eqref{eq:global-congruence}.

The next estimate controls the cost of restoring the outgoing wave over
a finite range of indices.

\begin{lemma}\label{lem:fourier-restoration}
Suppose that $\ell$ is an integer with $\ell\ge1$,
$0<h<1/100$, $0\le\Delta<1/100$, and
\[
 y_j=H_j+\mathbf1_{\{j\ge\ell\}}\chi_j,
 \qquad |\chi_j|\le\Delta\quad(j\ge\ell),
\]
\[
 E=\sum_{j\in\mathbb Z}V(H_j)<\infty,
 \qquad |H_j|\le hr^{j-\ell+1}\quad(j\ge\ell).
\]
Put $C_0=2\pi^2$ and $A=2C_0hr/(1-r)$. Then, for every positive integer
$J$ and every $0\le n\le J$,
\begin{equation}\label{eq:prefix-action}
 \sum_{j\le n}V(y_j)\le
 B_J:=E+A\Delta+C_0(J+1)\Delta^2.
\end{equation}
\end{lemma}

\begin{proof}
On $[-1/4,1/4]$, one has $V''\le2C_0$ and
$|V'(u)|\le2C_0|u|$. For $j\ge\ell$,
\[
 |H_j|+|\chi_j|\le hr+\Delta<2/100<1/4.
\]
Taylor's theorem on this interval yields
\[
 V(H_j+\chi_j)
 \le V(H_j)+2C_0|H_j|\Delta+C_0\Delta^2.
\]
For $n\ge\ell$, sum over $\ell\le j\le n$. The mixed terms contribute
at most $A\Delta$, and there are at most $J+1$ quadratic terms.
Adding the unchanged coordinates and using
$\sum_{j\le n}V(H_j)\le E$ proves \eqref{eq:prefix-action}.
For $n<\ell$, the prefix is unchanged and its action is at most $E$.
\end{proof}

No comparison between $J$ and $\ell$ is needed: when $n<\ell$,
the prefix is unchanged, and when $n\ge\ell$, the restoration
estimate applies. 
We can now choose a trajectory and a finite family of Fourier
frequencies for which the entropy gain exceeds the grid and separation costs.

\begin{proposition}
\label{prop:abstract-closure}
Fix a quartic Salem number $\beta$ with $\beta^2>2$. Suppose that there
is a sequence of localized trajectories satisfying
\[
 Y_j^{(k)}\equiv c_k\beta^j\pmod{\mathbb Z},\qquad
 Y_j^{(k)}=H_j^{(k)}+\mathbf1_{\{j\ge\ell_k\}}\chi_j^{(k)},
\]
\[
\begin{aligned}
 |\chi_j^{(k)}|&\le\Delta_k,\\
 |H_j^{(k)}|
 &\le hr^{j-\ell_k+1}\quad(j\ge\ell_k),\\
 0<c_*&\le|c_k|\le h<1/100,
\end{aligned}
\]
where $h,c_*$ are independent of $k$ and $\ell_k\ge1$.
Assume that 
$$
E_k=\sum_{j\in\mathbb Z}V(H_j^{(k)})<\infty
$$ 
and that
\[
 \Delta_k>0,\qquad \Delta_k\longrightarrow0,
 \qquad E_k=o\!\left(\log\frac1{\Delta_k}\right).
\]
Then $H_N<N\log\beta$ for some fixed finite positive integer $N$.
In particular, $\dim_{\mathrm H}\nu_{\beta^{-1}}<1$.
\end{proposition}

\begin{proof}
For all sufficiently large $k$, we may apply
Lemma~\ref{lem:fourier-restoration}. Put
\[
 \begin{gathered}
 x_k=\log(1/\Delta_k),\qquad
 J_k=\lceil\Delta_k^{-1}\rceil,\\
 B_k=E_k+A\Delta_k+C_0(J_k+1)\Delta_k^2,
 \qquad a_k=\min\{1/2,e^{-B_k/2}\}.
 \end{gathered}
\]
Here $C_0$ and $A$ are fixed. For $\Delta_k\le1$,
\begin{equation}\label{eq:J-range}
 e^{x_k}\le J_k\le e^{x_k}+1\le2e^{x_k},
\end{equation}
and, with $\Lambda=A+3C_0$,
\[
 B_k\le E_k+\Lambda\Delta_k.
\]
Since $E_k/x_k\to0$ and $\Delta_k\to0$, for all sufficiently large $k$
we have
\begin{equation}\label{eq:B-bound}
 B_k\le x_k/2+1.
\end{equation}
Since $B_k\ge0$, the definition of $a_k$ gives
\begin{equation}\label{eq:a-bound}
 a_k^2\ge\tfrac14e^{-B_k}.
\end{equation}

Let $s_k$ be the least nonnegative integer for which
$|c_k|\beta^{s_k}\ge1$. The uniform lower bound $c_*$ gives a fixed
finite integer $S$ with $s_k\le S$ for every $k$. For $J_k\ge s_k$,
set
\[
 M_k=1+\left\lfloor\frac{J_k-s_k}{2}\right\rfloor,
 \qquad
 t_{j+1}=\pi|c_k|\beta^{s_k+2j}\quad(0\le j<M_k).
\]
These frequencies start at or above $\pi$ and have common ratio
$\beta^2>2$. Every selected index is at most $J_k$, so
\begin{equation*}
 \beta^{-J_k}t_{M_k}\le\pi|c_k|\le\pi h<1.
\end{equation*}
Equations \eqref{eq:fourier-product} and \eqref{eq:prefix-action} give
$|\widehat\nu(t_j)|\ge e^{-B_k/2}\ge a_k$.
Thus Lemma~\ref{lem:entropy} applies with $N=J_k$ and $a=a_k$.

If $J_k\ge2S$, then $M_k\ge J_k/4$. Equations
\eqref{eq:J-range}, \eqref{eq:B-bound}, and \eqref{eq:a-bound}
consequently give
\begin{equation*}
 \frac12M_ka_k^2\ge\frac{J_ke^{-B_k}}{32}
 \ge\frac{e^{x_k/2}}{32e}.
\end{equation*}
On the other hand, for the fixed positive constant
$C'=2\log2+C_\beta$,
\begin{equation*}
 2\log J_k+C_\beta\le2x_k+C'.
\end{equation*}
Since $x_k\to\infty$, the gain $e^{x_k/2}/(32e)$ eventually
exceeds $2x_k+C'$. We may therefore fix a finite index $k_*$ for
which all the preceding conditions hold and
\[
 \tfrac12M_{k_*}a_{k_*}^2>2\log J_{k_*}+C_\beta.
\]
Set $N=J_{k_*}$. Lemma~\ref{lem:entropy} gives $H_N<N\log\beta$.
Applying Lemma~\ref{lem:block} at this fixed scale proves the strict
dimension bound.
\end{proof}

\begin{proof}[Proof of Theorem~\ref{thm:main}]
If $\beta>2$, then $H_1=\log2<\log\beta$.
Lemma~\ref{lem:block}, applied with $N=1$, gives
\[
 \dim_{\mathrm H}\nu_{\beta^{-1}}
 \le\frac{\log2}{\log\beta}<1.
\]
Since $2$ is not a quartic Salem number, it remains to consider
$1<\beta<2$.

By Proposition~\ref{prop:quartic-classification}, every quartic
Salem number in this range is a root of one of the two polynomials
in the introduction. Lemma~\ref{lem:roots} shows that both satisfy
$\beta^2>2$. For either parameter,
Proposition~\ref{prop:hierarchy} gives a sequence of reciprocal
localized orbits with
\[
 \Delta_k>0,\qquad \Delta_k\longrightarrow0,\qquad
 E_k=o\!\left(\log\frac1{\Delta_k}\right).
\]
These orbits have a common left endpoint. After translating this
endpoint to zero, the bound \eqref{eq:c-bound} and the identities
\eqref{eq:global-congruence}--\eqref{eq:localized-structure}
show that the remaining hypotheses of
Proposition~\ref{prop:abstract-closure} are satisfied.
We may therefore apply that proposition to obtain a finite
positive integer $N$ such that
\[
 H_N<N\log\beta.
\]
Applying Lemma~\ref{lem:block} once more yields
\[
 \dim_{\mathrm H}\nu_{\beta^{-1}}
 \le\frac{H_N}{N\log\beta}<1,
\]
as required.
\end{proof}

\appendix
\section{Explicit finite configurations}\label{app:explicit}

We now apply the Fourier--Shannon criterion to two explicit configurations
and prove Theorem~\ref{thm:quantitative}. The seeds and returns follow
\cite[Appendix~A]{CFMZ}; the accompanying exact arithmetic programs
recompute the required numerical bounds. For the first parameter, we
start with $U_1^3$ and apply two long mirrors. For the second, we use a
finite Laurent-polynomial modification of $U_2^4$, with an index
translation, and estimate it through a Green sequence and a direct
difference. Each construction supplies a finite Fourier input of the
type used in the main proof.

Proposition~\ref{prop:hierarchy} uses a return schedule to prove
existence. Here we choose the delays directly and check the hypotheses
of the finite estimates. We retain the notation
$\kappa_\lambda=1/P'(\lambda)$ and
$\mathcal C(t)=\pi^2\sec^2(\pi t)$.

\subsection{The case $P=P_1$}
\label{app:p1}

For $P=P_1$, write
\[
 u=\frac{1+\sqrt{13}}2,\qquad
 v=\frac{1-\sqrt{13}}2,\qquad \alpha=u+1=U_1(\beta).
\]
Take the power seed $K_{\rm s}=U_1^3$ from Lemma~\ref{lem:seed}
and the core $[-11,15]$. The expanding coordinates are
\begin{equation}\label{app:p1-seed-orbit}
 \xi_t=-\kappa_\beta\beta^{t-1}\alpha^3,
 \qquad Y_t=\operatorname{can}(\xi_t).
\end{equation}
The elementary Laurent-polynomial identity
\begin{equation}\label{app:p1-identity}
 XU_1(X)^3-(10X^2+13X+10)
 =X^{-2}(X^2+4X+1)P_1(X)
\end{equation}
relates this seed to the polynomial used in the computation. At every
root of $P_1$, the polynomial $10X^2+13X+10$ agrees with $XU_1^3$.
Thus its coordinate with index $i$ is the coordinate
\eqref{app:p1-seed-orbit} with index $t=i+1$. The program uses the
core $[-12,14]$ in the $i$ coordinate. This translation preserves the
action, amplitudes, and tail coefficients.

Here $j_*=4$. Localize by setting $H_t=Y_t$ for $t\le15$ and
$H_t=\xi_{4-t}$ for $t\ge16$. The exact formulas
$(u+1)(v+1)=-1$ and $|1-\zeta^2|^2=u$ give
\begin{equation*}
 h_0=|\kappa_\beta|r^{12}\alpha^3,
 \qquad
 \Delta_0=\frac{2\alpha^{-3}}{\sqrt{13}\sqrt u}.
\end{equation*}
Outward rational evaluation verifies
\begin{equation*}
 E_0<\frac{61}{5},\qquad
 \vartheta_0>\frac{39}{500},\qquad
 h_0<\frac1{100},\qquad
 0<\Delta_0<\frac{203}{20000}.
\end{equation*}
Here $E_0$ includes both localized real tails. Their total action is
at most
\[
 2\mathcal C(h_0r)\frac{h_0^2r^2}{1-r^2}.
\]
The core has $27$ coordinates. The first four expanding coordinates
have moduli at most $\beta^3h_0<8/100$, so no centering is needed at
the left boundary and the recurrence is homogeneous there. Beyond
the right boundary, $h_0r^{t-15}+\Delta_0<1/2$ shows that the sum of
the stable and neutral terms is already canonical. These bounds
verify the two infinite tails.

The nonzero entries of $P_1(S)Y$ in the $t$ coordinate are
\[
\begin{aligned}
 (P_1(S)Y)_{-7}&=1,
 & (P_1(S)Y)_{-4}&=-1,\\
 (P_1(S)Y)_4&=-1,
 & (P_1(S)Y)_7&=1.
\end{aligned}
\]
The program checks the corresponding indices $-8,-5,3,6$ in its $i$
coordinate. The boundary identities show that the forcing vanishes
outside the core.

Apply Lemma~\ref{lem:mirror} with
\begin{equation*}
 L_1=124,\qquad L_2=489637838.
\end{equation*}
Both delays are even and exceed the core lengths to which they are
applied. The certified returns are
\[
 |1-\zeta^{L_1}|<2507\cdot10^{-6},\qquad
 |1-\zeta^{L_2}|<3736\cdot10^{-12}.
\]
At each stage the program checks the mirror hypotheses and applies
\eqref{eq:core-derivative-bound} and \eqref{eq:mirror-action}.
The actual real-tail coefficient decreases, so we keep the seed value
$h_0$ as an upper bound. The final core is
\[
 [-11,489637977],\qquad \ell_1=489637989.
\]
After translating its left endpoint to zero, the certified bounds are
\begin{equation*}
 E<50.27,\qquad
 \Delta<9.498\cdot10^{-14},\qquad
 |H_j|<\frac12-\frac{39}{500},
\end{equation*}
and
\begin{equation}\label{app:p1-tail}
 |H_j|\le\frac1{100}
 r^{\operatorname{dist}(j,[0,\ell_1-1])}
 \qquad(j\notin[0,\ell_1-1]).
\end{equation}
The expanding left coefficient is nonzero and satisfies
\begin{equation}\label{app:p1-left}
 |c_1|=|\kappa_\beta|r^{12}\alpha^3
              (1-r^{124})(1-r^{489637838}).
\end{equation}
Each mirror multiplies the actual real-tail coefficient by
$1-r^L<1$, which gives the common bound \eqref{app:p1-tail}.

\subsection{The case $P=P_2$}
\label{app:p2}

For $P=P_2$, define the bilateral Green sequence
\begin{equation*}
 g_i=\begin{cases}
  -\kappa_\beta\beta^{i-1},&i\le0,\\
  \kappa_r r^{i-1}
     +2\operatorname{Re}(\kappa_\zeta\zeta^{i-1}),&i\ge1.
 \end{cases}
\end{equation*}
We have $P(S)g=e^{(0)}$, where $e^{(0)}$ equals one at index zero
and zero elsewhere. To see this, note that both branches are
homogeneous and their difference at index $i$ is $T_{i-1}$ from
\eqref{eq:trace-int}. Since $T_0=T_1=T_2=0$ and $T_3=1$, the forcing
values at the possible junction indices $-3,-2,-1,0$ are $0,0,0,1$.
The same identity gives
$g_i\equiv-\kappa_\beta\beta^{i-1}\pmod{\mathbb Z}$ for every $i$.

The following estimate allows a core without reflection symmetry and an odd delay.
We will use it for the final difference in the second construction.

\begin{lemma}\label{app:direct-difference}
Let $\ell\ge1$ be an integer, and let $Y$ be a canonical real sequence
with nonzero integer forcing
$P(S)Y$ supported in $[0,\ell-1]$. Suppose its core action and maximum
satisfy
\[
 E_c=\sum_{i=0}^{\ell-1}V(Y_i)<\infty,
 \qquad |Y_i|\le a_*\quad(0\le i<\ell),
\]
and its tails have the form
\[
 Y_i=a_-\beta^i\quad(i<0),\qquad
 Y_i=2\operatorname{Re}(\gamma\zeta^i)+a_+r^{i-\ell}
       \quad(i\ge\ell),
\]
where $|a_-|,|a_+|\le\tau$ and $\Delta=2|\gamma|>0$.
For an integer $L\ge\ell$, put $g=L-\ell$ and
\[
 A_r(m)=\left(\frac{1-r^{2m}}{1-r^2}\right)^{1/2},\qquad
 A_\zeta(m)=
 \left(\min\left\{m,\frac m2+\frac1{|1-\zeta^2|}\right\}\right)^{1/2}.
\]
Define
\begin{align*}
 \alpha_1&=\tau r^{g+1}A_r(\ell),&
 \alpha_2&=\Delta A_\zeta(\ell)+\tau r^gA_r(\ell),\\
 \alpha_g&=\Delta A_\zeta(g)+\tau(1+r)A_r(g),&
 m_c&=a_*+\Delta+\tau r^g,\qquad m_g=2\Delta+2\tau.
\end{align*}
If $m_c,m_g<1/2$, then $Y_i^+=Y_i-Y_{i-L}$ is canonical, has nonzero
integer forcing supported in $[0,\ell+L-1]$, and has outgoing
amplitude $\Delta^+=\Delta|1-\zeta^{-L}|$. Its core action satisfies
\begin{align}
 E_c^+\le{}&2E_c+
 2\sqrt{\mathcal C(m_c)E_c}\,(\alpha_1+\alpha_2)\notag\\
 &+\mathcal C(m_c)(\alpha_1^2+\alpha_2^2)
 +\mathcal C(m_g)\alpha_g^2. \label{app:direct-action}
\end{align}
The new real-tail coefficients still have moduli at most $\tau$.
\end{lemma}

\begin{proof}
On the first core $0\le i<\ell$, the perturbation from $Y_i$ is
$-a_-\beta^{i-L}$ and has Euclidean norm at most $\alpha_1$.
On the translated core $i=L+t$, $0\le t<\ell$, compare $Y_i^+$ with
$-Y_t$. The perturbation is the old right tail $Y_{L+t}$, whose
Euclidean norm is at most $\alpha_2$.
The real tails give geometric sums. For the neutral term, we use
\begin{equation}\label{app:cosine-square-sum}
 \sum_{j=0}^{m-1}\cos^2(\theta+j\arg\zeta)
 \le\min\left\{m,\frac m2+\frac1{|1-\zeta^2|}\right\}.
\end{equation}
For the second bound in \eqref{app:cosine-square-sum}, write
$\cos^2 t=(1+\cos2t)/2$ and sum the progression in $\zeta^2$.
The nonconstant part has modulus at most $1/|1-\zeta^2|$.
The first bound follows term by term.

The Taylor segments on both cores lie in $[-m_c,m_c]$.
By \eqref{eq:potential-basic}, the derivative vectors have norm at
most $2\sqrt{\mathcal C(m_c)E_c}$. Apply Taylor's theorem and
Cauchy--Schwarz, using the evenness of $V$. The first three terms
on the right side of \eqref{app:direct-action} then bound the action
on the two cores.

In the gap, write $i=\ell+t$ with $0\le t<g$. The difference is
\[
 2\operatorname{Re}(\gamma\zeta^{\ell+t})
       +a_+r^t-a_-r^{g-t}.
\]
Its Euclidean norm is at most $\alpha_g$ and its pointwise modulus
is at most $\Delta+2\tau\le m_g$. The estimate
$V(u)\le\mathcal C(m_g)u^2$ gives the remaining term in
\eqref{app:direct-action}. The three regions partition the new core.

Outside the core, direct substitution gives
\begin{equation*}
 a_-^+=(1-r^L)a_-,\qquad
 a_+^+=(r^L-1)a_+,\qquad
 \gamma^+=(1-\zeta^{-L})\gamma.
\end{equation*}
Thus the real-tail coefficients decrease in modulus and the new
neutral amplitude is as stated. The left exterior has modulus at most
$\tau r$, and the right exterior at most $2\Delta+\tau\le m_g$.
The core, gap, and exterior bounds therefore place every coordinate
of $Y^+$ strictly between $-1/2$ and $1/2$. Finally,
$P(S)Y^+=P(S)Y-(P(S)Y)_{\cdot-L}$; the two forcing supports are
disjoint since $L\ge\ell$, so their difference is nonzero and integer.
\end{proof}

Define the integral polynomials
\begin{align*}
 K_{\rm b}(X)&=1-X+X^2-X^3+X^4-X^7+X^8-X^9+X^{10}-X^{11},\\
 K_{\rm seed}(X)&=K_{\rm b}(X)(1-X^{2673})
                 =\sum_{j=0}^{2684}k_jX^j. 
\end{align*}
The finite Green configuration is
\begin{equation*}
 Y_i^{\rm s}=\sum_{j=0}^{2684}k_jg_{i-100-j},
 \qquad \ell_{\rm s}=2885.
\end{equation*}
Its forcing is $k_{i-100}$, supported in
$[100,2784]\subset[0,2884]$. The two real-tail coefficients and
neutral coefficient are exactly
\begin{equation}\label{app:p2-seed-tails}
 \begin{split}
 a_-&=-\kappa_\beta r^{101}K_{\rm seed}(r),\\
 a_+&=\kappa_r\sum_{j=0}^{2684} k_jr^{\ell_{\rm s}-101-j},\\
 \gamma&=\kappa_\zeta\zeta^{-101}K_{\rm seed}(\zeta^{-1}).
 \end{split}
\end{equation}
The certificate evaluates the $2885$ core coordinates and verifies
\begin{equation*}
\begin{aligned}
 |a_-|,|a_+|&<10^{-24},
 & 0<\Delta=2|\gamma|&<47\cdot10^{-7},\\
 E_c&<24,
 & \max_{0\le i<2885}|Y_i^{\rm s}|&<0.394.
\end{aligned}
\end{equation*}
Together with \eqref{app:p2-seed-tails}, these bounds show that both
infinite tails are canonical. Thus the seed is canonical at every
integer index.

To relate this configuration to the power seed $U_2^4$, use the identity
\begin{equation}\label{app:p2-identity}
 \begin{split}
 K_{\rm b}(X)-X^5(1-X)U_2(X)^4
 =P_2(X)(1+X^2-X^3+X^4-X^5-X^7).
 \end{split}
\end{equation}
Put
\[
 \mathcal K_2(X)=U_2(X)^4(1-X^{-1})(1-X^{-2673}).
\]
At the roots of $P_2$, \eqref{app:p2-identity} implies
$K_{\rm seed}(X^{-1})=X^{-5}\mathcal K_2(X)$.
The Green congruence now gives
\begin{equation}\label{app:p2-power-seed-orbit}
 Y_i^{\rm s}=
 \operatorname{can}\!\left(
      -\kappa_\beta\beta^{i-106}\mathcal K_2(\beta)\right).
\end{equation}
Thus the trajectory is generated by $\mathcal K_2$, with orbit
index $i-105$. We treat the factors of lengths $1$ and $2673$ as
algebraic preprocessing. For the next step we use the direct-difference
estimate, keeping the $2885$-coordinate Green core and its tail
indexing. This avoids any symmetry requirement on the core in
\eqref{app:p2-power-seed-orbit}.

Apply Lemma~\ref{app:direct-difference} once, with
\begin{equation*}
 \tau=10^{-24},\qquad L=1846082221.
\end{equation*}
The program verifies the return bound
$|1-\zeta^L|<2365\cdot10^{-13}$ and the lemma's pointwise hypotheses
before applying the action estimate. The new core length is
\[
 \ell_2=2885+1846082221=1846085106,
\]
and the forcing lies in the disjoint union
\[
 [100,2784]\ \cup\ [1846082321,1846085005]
 \subset[0,\ell_2-1].
\]
Localize the new sequence by
\[
 H_i=Y_i^+\quad(i<\ell_2),\qquad
 H_i=a_+^+r^{i-\ell_2}\quad(i\ge\ell_2).
\]
This removes $2\operatorname{Re}(\gamma^+\zeta^i)$ from the final
right tail. The direct-difference estimate places the new core action
strictly below $47.626735096901$. We must also include the two localized tails.
The left tail contributes at most
$\mathcal C(\tau)\tau^2r^2/(1-r^2)$, while the right tail, whose
first coefficient occurs at $i=\ell_2$, contributes at most
$\mathcal C(\tau)\tau^2/(1-r^2)$. Adding these contributions gives the full localized action bound
\begin{equation}\label{app:p2-full-action}
 E\le E_c^++\frac{2\mathcal C(\tau)\tau^2}{1-r^2}.
\end{equation}
The added term is certified below $2.75\cdot10^{-47}$. The resulting
full bounds are
\begin{equation*}
 E<47.626736,\qquad
 \Delta<1.105305\cdot10^{-15},\qquad |H_i|<2/5,
\end{equation*}
and
\begin{equation*}
 |H_i|\le10^{-23}r^{\operatorname{dist}(i,[0,\ell_2-1])}
 \qquad(i\notin[0,\ell_2-1]).
\end{equation*}
Indeed, $r>1/2$ and the first right-tail coefficient satisfies
$|a_+^+|<10^{-24}<10^{-23}r$. The program also checks
$|a_+^+|<10^{-23}r$ directly, with the right-tail index as above.
For the left tail, use $|a_-^+|<10^{-24}$.

The nonzero left coefficient after the difference has modulus
\begin{equation}\label{app:p2-left}
 |c_2|=|\kappa_\beta|r^{101}K_{\rm b}(r)
           (1-r^{2673})(1-r^{1846082221}).
\end{equation}
Here $K_{\rm b}(r)>0$, either by \eqref{app:p2-identity} or by grouping
its terms as positive differences.

\subsection{The certified finite inputs}\label{app:finite-inputs}

Table~\ref{tab:short-stages} gives the recomputed strict upper bounds.
For the first parameter, each action bound includes both localized
tails. For the second, the seed entry bounds the core action and the
final entry includes the tails through \eqref{app:p2-full-action}.
\begin{table}[htbp]
\centering
\begingroup
\small
\setlength{\tabcolsep}{4pt}
\begin{tabular}{@{}llcc@{}}
\toprule
Parameter & Stage
 & \shortstack{Action\\upper bound}
 & \shortstack{Amplitude\\upper bound}\\
\midrule
$P_1$ & Seed $U_1^3$
 & $12.143280$ & $1.014602\cdot10^{-2}$\\
 & First mirror
 & $24.352783$ & $2.542925\cdot10^{-5}$\\
 & Second mirror
 & $50.269214$ & $9.497533\cdot10^{-14}$\\
\midrule
$P_2$ & Preprocessed $U_2^4$ seed
 & $23.705603$ & $4.674807\cdot10^{-6}$\\
 & Direct difference, with tails
 & $47.626736$ & $1.105305\cdot10^{-15}$\\
\bottomrule
\end{tabular}
\endgroup
\caption{Certified bounds at the finite stages. The full localized
action of the final configuration is used in the Fourier estimate.}
\label{tab:short-stages}
\end{table}

After translating the first core to start at zero, the inputs used below
are as follows.
\begin{table}[htbp]
\centering
\begin{tabular}{@{}lcc@{}}
\toprule
Quantity & $P_1$ & $P_2$\\
\midrule
Core length $\ell$ & $489637989$ & $1846085106$\\
Full action bound $E_{\max}$ & $50.27$ & $47.626736$\\
Amplitude bound $\Delta$ & $9.498\cdot10^{-14}$ & $1.105305\cdot10^{-15}$\\
Tail prefactor $h$ & $10^{-2}$ & $10^{-23}$\\
Least $s\ge0$ with $|c|\beta^s\ge1$ & $10$ & $105$\\
\bottomrule
\end{tabular}
\caption{Rational inputs to the finite Fourier certificate.
The convention is $|H_j|\le hr^{j-\ell+1}$ for $j\ge\ell$,
so the first right-tail bound is $hr$.}
\label{tab:finite-inputs}
\end{table}

To use the same frequency count for both configurations, we note that
\begin{equation}\label{app:common-c-bound}
 2^{-112}<|c_i|<1\qquad(i=1,2).
\end{equation}
The program checks these inequalities. They also follow directly
from the coefficient formulas. For $P_1$, $|P_1'(\beta)|<49$, $\alpha>3$, $r>1/2$, and
the two factors $1-r^{L_j}$ in \eqref{app:p1-left} exceed $1/2$.
Thus $|c_1|>27/(49\cdot2^{14})>2^{-112}$.
For $P_2$, $|P_2'(\beta)|<62$ and
\[
 K_{\rm b}(r)=(1-r)+(r^2-r^3)+(r^4-r^7)
                  +(r^8-r^9)+(r^{10}-r^{11})>2/5.
\]
The two remaining factors in \eqref{app:p2-left} exceed $1/2$, so
$|c_2|>2^{-101}/620>2^{-112}$.
The upper bounds follow from the certified real-tail coefficients,
which decrease under each difference.
Since $\beta^2>2$, these lower bounds already imply $s\le224$;
the certificate uses the looser common bound $s\le256$.

The first configuration satisfies the Fourier congruence by its
definition as a centered expanding orbit. For the second, use
\eqref{app:p2-power-seed-orbit} and take the direct difference.
Alternatively, apply Lemma~\ref{lem:recurrence-congruence} to the
integer forcing and the exact left tail. Thus both action bounds
apply to the Bernoulli Fourier orbit at every integer index.

\section{Finite Fourier bounds and the entropy deficit}
\label{app:scalar}

We next restore the neutral term over a finite range of orbit indices.
Retaining its oscillatory square sum gives the following bound.

\begin{lemma}
\label{app:sharp-restoration}
Suppose that
\[
 y_j=H_j+\mathbf1_{\{j\ge\ell\}}\chi_j,\qquad
 E=\sum_{j\in\mathbb Z}V(H_j)<\infty,\qquad
 |H_j|\le hr^{j-\ell+1}\quad(j\ge\ell),
\]
where $\ell\ge1$ is an integer and
\[
 \chi_j=\delta\cos(\theta+(j-\ell)\phi),\qquad
 \zeta=e^{i\phi},\qquad 0\le\delta\le\Delta,
 \qquad hr+\Delta\le1/100.
\]
Let $d=|1-\zeta^2|>0$. For every integer $J\ge0$ and every
integer $n$ with $0\le n\le J$,
\begin{equation*}
 \sum_{j\le n}V(y_j)
 \le E+\frac{20\Delta hr}{1-r}
       +10\Delta^2\left(\frac{J+1}{2}+\frac1d\right).
\end{equation*}
\end{lemma}

\begin{proof}
For $|u|\le1/100$, $|\sin t|\le|t|$ and $\pi<22/7$ give
\[
 \cos^2(\pi u)\ge1-\pi^2u^2>
 \frac{122379}{122500},\qquad
 V''(u)<\frac{968}{49}\frac{122500}{122379}
       =\frac{2420000}{122379}<20.
\]
Since $V'(0)=0$, we also have $|V'(u)|\le20|u|$ on this interval.
Every segment from $H_j$ to $H_j+\chi_j$, with $j\ge\ell$, stays in
the same interval. Taylor's theorem gives
\[
 V(H_j+\chi_j)
 \le V(H_j)+20|H_j|\Delta+10\chi_j^2.
\]
If $n<\ell$, the conclusion follows immediately from the definition
of $E$. Otherwise let $m=n-\ell+1$. The mixed terms sum to at most
$20\Delta hr/(1-r)$. The exact geometric sum gives
\[
 \sum_{q=0}^{m-1}\cos^2(\theta+q\phi)
 =\frac m2+\frac12\operatorname{Re}
       \left(e^{2i\theta}\frac{1-\zeta^{2m}}{1-\zeta^2}\right)
 \le\frac m2+\frac1d
 \le\frac{J+1}{2}+\frac1d.
\]
Finally, $\sum_{j\le n}V(H_j)\le E$. Since $E$ includes the stable
tail, the bound requires no further $h^2$ term.
\end{proof}

For the two inputs in Table~\ref{tab:finite-inputs}, put
\begin{equation*}
 J_1=2\cdot10^{25},\qquad J_2=10^{26},\qquad N_i=J_i+3.
\end{equation*}
With the strict rational bounds $E_{\max}$ and $\Delta$ from the
table, set
\begin{equation*}
 A_J=E_{\max}+\frac{20\Delta hr}{1-r}
       +10\Delta^2\left(\frac{J+1}{2}
                              +\frac1{|1-\zeta^2|}\right).
\end{equation*}
The certificate verifies $hr+\Delta<1/100$ and gives the strict bounds
\begin{equation*}
 A_{J_1}<51.172120040000027<51.18,
 \quad
 A_{J_2}<47.627346849571513<47.63.
\end{equation*}
Since the actual full action is strictly below $E_{\max}$,
\eqref{eq:fourier-product} and Lemma~\ref{app:sharp-restoration} imply
\begin{equation}\label{app:fourier-square-lower}
 |\widehat\nu_{\beta_i^{-1}}(\pi|c_i|\beta_i^n)|^2
 >e^{-A_{J_i}}\qquad(0\le n\le J_i).
\end{equation}
Thus each configuration gives a lower bound at every frequency in the
finite set $\{\pi|c_i|\beta_i^n:n=0,1,\ldots,J_i\}$ ($i=1,2$).

Choose the rational modulus bounds
\begin{equation}\label{app:a-values}
 a_1=\frac3{4\cdot10^{11}},\qquad
 a_2=\frac9{2\cdot10^{11}}.
\end{equation}
We verify them by integer comparisons. Indeed,
\[
 e<\sum_{j=0}^{6}\frac1{j!}+\frac8{7\cdot7!}
   =\frac{31967}{11760}<\frac{87}{32},
\]
and the program verifies the exact integer inequalities
\begin{align}
 87^{2559}3^{100}
 &<32^{2559}(4\cdot10^{11})^{100},\notag\\
 87^{4763}9^{200}
 &<32^{4763}(2\cdot10^{11})^{200}. \label{app:exponential-integers}
\end{align}
Taking the appropriate roots shows
$e^{51.18}a_1^2<1$ and $e^{47.63}a_2^2<1$. Consequently every
Fourier modulus in \eqref{app:fourier-square-lower} is strictly greater
than the corresponding $a_i$.

We can now count the frequencies and apply the entropy estimate.

\begin{proof}[Proof of Theorem~\ref{thm:quantitative}]
For either parameter let $s$ be the least nonnegative integer such that
$|c|\beta^s\ge1$, and select
\[
 t_{j+1}=\pi|c|\beta^{s+2j},
 \qquad 0\le j\le\left\lfloor\frac{J-s}{2}\right\rfloor.
\]
The first frequency is at least $\pi$, the ratio is
$Q=\beta^2>25/9$, and all selected exponents are at most $J$.
By \eqref{app:common-c-bound} and $s\le256$, their number is at least
\begin{equation*}
 M_0=1+\left\lfloor\frac{J-256}{2}\right\rfloor
    =\frac J2-127.
\end{equation*}
For the scale $N=J+3$ retained in the numerical certificate,
\[
 \beta^{-N}t_{\max}\le\pi|c|\beta^{-3}<1,
\]
using $|c|<1$, $\beta>5/3$, and $\pi<22/7<125/27$.
The two rational $a_i$ in \eqref{app:a-values} are less than $1/9$,
so Corollary~\ref{cor:entropy-sharp} applies.

The exact comparison
\begin{equation}\label{app:log-N-integer}
 163^{60}>N\,60^{60},\qquad
 e>\sum_{j=0}^{5}\frac1{j!}=\frac{163}{60},
\end{equation}
holds for both values of $N$ and proves $\log N<60$.
Hence the entropy penalty is at most $2\log N+4<125$, and
\begin{equation}\label{app:entropy-deficit}
 H_N<N\log\beta-D,
 \qquad D=\frac12M_0a^2-125.
\end{equation}
The resulting positive rationals are
\begin{equation*}
 D_1=\frac{625}{4}-\frac{1143}{32\cdot10^{22}}>156,
 \qquad
 D_2=50500-\frac{10287}{8\cdot10^{22}}>50499.
\end{equation*}
In particular, the finite entropy bounds are
$H_{N_1}<N_1\log\beta_1-156$ and
$H_{N_2}<N_2\log\beta_2-50499$.

The certified algebraic root intervals satisfy
\[
 \beta_1^{20}<(163/60)^{11}<e^{11},\qquad
 \beta_2^{25}<(163/60)^{16}<e^{16}.
\]
Hence $\log\beta_1<11/20$ and $\log\beta_2<16/25$.
The final exact rational comparisons are
\begin{equation}\label{app:dimension-integers}
 \frac{D_1}{(J_1+3)(11/20)}>1.4\cdot10^{-23},\qquad
 \frac{D_2}{(J_2+3)(16/25)}>7.8\cdot10^{-22}.
\end{equation}
Combining \eqref{app:entropy-deficit} with
Lemma~\ref{lem:block} gives
\[
 \dim_{\mathrm H}\nu_{\beta^{-1}}
 \le\frac{H_N}{N\log\beta}
 <1-\frac{D}{N\log\beta}.
\]
The bounds on $\log\beta$ and \eqref{app:dimension-integers} prove
the two claimed upper bounds.
\end{proof}

\section{Exact arithmetic and reproduction}\label{app:arithmetic}

The companion archive contains the verification programs and three
deterministic reports for the upper-bound certificates, algebraic
seeds, and entropy bounds. The \texttt{README.md} file describes
the contents of the archive and gives further instructions for
running the programs. All finite checks can be reproduced by
running the following command from the archive root:
\begin{verbatim}
python3 -S verification/verify_all.py
\end{verbatim}
The program checks the archive manifest, recomputes the three
mathematical reports, and compares them byte for byte with the
supplied versions. It uses only Python's standard library.
All acceptance conditions are explicit and remain active under
Python optimization, without relying on optional \texttt{assert}
statements.

The numerical verification uses rational interval arithmetic
with outward rounding. The fixed denominator is $S=10^{200}$,
so every interval has endpoints in $S^{-1}\mathbb Z$.
Addition is exact on this grid. Multiplication and division
use the extreme values at the interval endpoints, followed
by outward integer rounding. Division is performed only when
the denominator interval excludes zero, while square roots
are enclosed using integer square roots of scaled endpoints.
These operations ensure that the resulting intervals contain
the exact values.

The algebraic roots are computed from the exact expressions
\[
 \beta=\frac{u+\sqrt{u^2-4}}2,\qquad
 r=\frac{u-\sqrt{u^2-4}}2,\qquad
 \zeta=\frac{v+i\sqrt{4-v^2}}2,
\]
with $(u,v)$ given by \eqref{eq:trace-roots}.
Complex numbers are enclosed by rectangular intervals, and
large powers are computed by repeated squaring.

The transcendental functions are also enclosed by rational
intervals. The bounds for $\pi$ follow from Machin's identity
\[
 \pi=16\arctan(1/5)-4\arctan(1/239),
\]
using $160$ terms of each alternating series and bounding the
remainder by the first omitted term. Sine and cosine are
evaluated by $100$-term Taylor sums for arguments of modulus
less than $2$, with remainder bounds $|x|^{201}/201!$ and
$|x|^{200}/200!$, respectively. To evaluate the potential,
let $c=\cos(\pi t)>0$ and $z=(1-c)/(1+c)$. Then
\begin{equation*}
 V(t)=4\sum_{j=0}^{159}\frac{z^{2j+1}}{2j+1}+R,
 \qquad 0\le R\le\frac{4z^{321}}{321(1-z^2)}.
\end{equation*}
This follows by expanding $-2\log c$ and bounding the remaining
positive terms by a geometric series. The program checks
$c>0$ and uses $c\le1$ when enclosing $z$. Thus the
transcendental estimates used in the verification have
rigorous rational bounds.

The finite checks are carried out separately for the two
polynomials. Canonical representatives are accepted only when
both endpoints of their enclosing intervals determine the
same centering integer and the centered interval lies strictly
inside $(-1/2,1/2)$. For $P_1$, the program evaluates the $27$
seed core potentials, verifies the finite forcing and the
bounds needed for the exterior formulas, and applies
\eqref{eq:mirror-action} twice. For $P_2$, it evaluates the
$2885$ seed core potentials, applies \eqref{app:direct-action},
and includes the stable tails through \eqref{app:p2-full-action}.
The program also checks the left coefficients, return
amplitudes, finite Fourier actions, and the inequalities in
\eqref{app:exponential-integers}, \eqref{app:log-N-integer},
and \eqref{app:dimension-integers}.

Two further checks complete the verification. The algebraic-seed
check establishes \eqref{app:p1-identity} and
\eqref{app:p2-identity} using exact integer Laurent-polynomial
arithmetic, including the index translations from the power
seeds to the finite configurations. The entropy-bounds check
takes the certified action, amplitude, and left-coefficient
intervals as inputs, independently isolates the algebraic
roots, and recomputes the rational entropy and dimension
bounds. The upper-bound verification establishes the finite
seed estimates and their propagation, while the entropy
constant follows from the analytic proof of
Corollary~\ref{cor:entropy-sharp}. The supplied reports are
used only for comparison after the computations have been
repeated.

The analytic estimates reduce the required computations to
the finite seed evaluations, outward interval powers, and
scalar inequalities. In particular, the verification does
not require enumerating the large cores or the $2^{N_i}$
words at the entropy scales $N_i$.

\sectionnotoc{Data and Code Availability}\label{sec:data-code}

The certificate data and verification programs supporting the explicit
Hausdorff-dimension bounds in Theorem~\ref{thm:quantitative} are openly
available on Zenodo~\cite{quartic_salem_hausdorff_verification}.
The archive contains the finite certificate inputs, verification code,
recorded outputs, and instructions for reproducing all finite checks.
The verification programs use exact rational interval arithmetic and
the Python standard library, as described in Appendix~\ref{app:arithmetic}.
The proof of Theorem~\ref{thm:main} is independent of this archive.

\sectionnotoc{Funding}

G.~C. was partially supported by the National Natural Science Foundation of China (NSFC), grant no.~12371126.
X.~F. was partially supported by the National Science and Technology Council, Taiwan, grant no.~114-2115-M-A49-003-MY3.

\sectionnotoc{AI Declaration}

The authors used OpenAI ChatGPT (GPT-6.0 Sol) for language editing, \LaTeX\ formatting, and limited assistance with local mathematical reasoning and numerical verification. 
The overall mathematical strategy and framework were developed by the authors, building on the work of~\cite{CFMZ}.
All mathematical content was written and checked by the authors who take full responsibility.

\sectionnotoc{Conflict of Interest}

The authors declare no conflicts of interest.

\end{document}